\documentclass[12pt]{amsart}
\usepackage{amsmath}
\usepackage{amsthm}
\usepackage{amssymb}
\usepackage{amscd}
\usepackage{amsfonts}
\usepackage{amsbsy}
\usepackage{epsfig}
\usepackage{graphicx}
\usepackage{tikz}
\usepackage{tikz-cd}
\usepackage{appendix}

\newtheorem{theorem}{Theorem}[section]
\newtheorem{corollary}[theorem]{Corollary}
\newtheorem{proposition}[theorem]{Proposition}
\newtheorem{lemma}[theorem]{Lemma}
\newtheorem{remark}[theorem]{Remark}

\newtheorem{definition}[theorem]{Definition}

\newtheorem{problem}{Problem}

\def\ie{{\em i.e.,} }
\def\eg{{\em e.g.} }

\newfont\bbf{msbm10 at 12pt}
\def\eps{\varepsilon}
\def\phi{\varphi}
\def\R{{\mathbb R}}

\def\N{{\mathbb N}}
\def\Z{{\mathbb Z}}

\def\cA{{\mathcal A}}
\def\cR{{\mathcal R}}

\def\E{{\mathcal E}}
\def\F{{\mathcal F}}

\def\U{{\mathcal U}}
\def\V{{\mathcal V}}
\def\W{{\mathcal W}}

\def\NL{N_L}
\def\NR{N_R}
\def\orb{\mbox{\rm orb}}
\def\theta{\vartheta}

\def\UIL{\underleftarrow\lim([0,1],T)}
\def\CUIL{\underleftarrow\lim([c_2,c_1],T)}
\def\ovl{\overleftarrow}
\def\ovr{\overrightarrow}

\def\eps{\varepsilon}

\def\dist{\mbox{\rm dist}}
\def\diam{\mbox{\rm diam}}

\begin{document}

\title{Folding points of unimodal inverse limit spaces}

\author{Lori Alvin, Ana Anu\v{s}i\'c, Henk Bruin, Jernej \v{C}in\v{c}}
\address[L.\ Alvin]{Department of Mathematics, Furman University, 3300 Poinsett Highway, Greenville, SC 29613, USA}
\email{lori.alvin@furman.edu}
\address[A.\ Anu\v{s}i\'c]{Departamento de Matem\'atica Aplicada, IME-USP, Rua de Mat\~ao 1010, Cidade Universit\'aria, 05508-090 S\~ao Paulo SP, Brazil}
\email{anaanusic@ime.usp.br}
\address[H.\ Bruin]{Faculty of Mathematics, University of Vienna,
Oskar-Morgenstern-Platz 1, A-1090 Vienna, Austria}
\email{henk.bruin@univie.ac.at}
\address[J.\ \v{C}in\v{c}]{National Supercomputing Centre IT4Innovations, Division of the University of Ostrava, Institute for Research and Applications of Fuzzy Modeling, 30. dubna 22, 70103 Ostrava, Czech Republic}
\email{jernej.cinc@osu.cz}
\thanks{AA was supported in part by Croatian Science Foundation under the project IP-2014-09-2285.
	She also gratefully acknowledges the Ernst Mach Stipend ICM-2017-06344 from the   \"Osterreichischer 
	Austauschdienst (OeAD).
	HB and J\v{C} were supported by the FWF stand-alone project P25975-N25. J\v{C} was partially supported by NSERC grant RGPIN 435518 and University of Ostrava grant IRP201824 “Complex topological
	structures”.
We gratefully acknowledge the support of the bilateral grant \emph{Strange Attractors and Inverse Limit Spaces},  \"Osterreichische
Austauschdienst (OeAD) - Ministry of Science, Education and Sport of the Republic of Croatia (MZOS), project number HR 03/2014.}
\date{\today}


\begin{abstract}
We study the properties of folding points and endpoints of unimodal inverse limit spaces.
We distinguish between non-end folding points and three types of end-points (flat, spiral and nasty)
and give conditions for their existence and prevalence. Additionally, we give a characterisation of tent inverse limit spaces for which the set of folding points equals the set of endpoints.
\end{abstract}

\maketitle

\section{Introduction}\label{sec:intro}
In 1967, Williams proved that hyperbolic one-dimensional attractors can be represented as inverse limits of maps on branched manifolds and that every point has a neighbourhood homeomorphic to the product of a Cantor set and an open arc \cite{Wi}. In this paper we study attractors which contain points that do not have such a neighbourhood (called {\em folding points}), the existence of which usually indicates the lack of hyperbolicity, or more generally, foliation of the space by unstable manifolds. For simplicity we conduct our study only for {\em unimodal inverse limits} which naturally arise as attractors of certain planar homeomorphisms with the action conjugated to the shift map, see \cite{BaMa2}. Every unimodal inverse limit contains folding points. Thus, the structure of (un)stable sets is considerably more complicated than in Williams' solenoids and it is still not completely understood, especially in cases when there exist spiral or nasty endpoints (see below).

Unimodal inverse limits link one dimensional dynamics to the dynamics of some well known planar systems, \eg H\'enon maps. It was shown in \cite{Hen} that for a dense set of parameters of maps in the H\'enon family, the attracting sets are homeomorphic to inverse limits of unimodal maps of the interval. In those cases, all but finitely many points are locally homeomorphic with the product of a Cantor set and an arc \cite{BaHo}, implying the pseudo-Anosov type system. 
On the other hand, it was shown in \cite{BBD} that for a dense $G_\delta$ set of parameters in the symmetric tent family, the inverse limit space is nowhere locally homeomorphic to the product of a Cantor set and an arc; in this case not only does every open set contain a homeomorphic copy of the entire space, but every set also contains a homeomorphic copy of every inverse limit space appearing in the tent family.

In addition to H\'enon attractors, Lozi~\cite{Loz} and Lorenz attractors~\cite{Lorenz} are prominent examples of chaotic attractors in dynamical systems. Although not all examples of these attractors arise as inverse limits of symmetric tent maps, we may still draw some parallels between unimodal inverse limits and H\'enon attractors, Lozi attractors and Poincar\'e sections of Lorenz attractors (for details on Poincar\'e sections of Lorenz attractors see \eg \cite{Gu}). For example, one can ask if all the types of inhomogeneities that arise in inverse limits of tent maps appear in these families of parametrised attractors as well.  This motivates us to first better understand the inhomogeneities from unimodal inverse limit spaces which are easier to approach.

Let  $T:[0,1]\to [0,1]$ be a unimodal map with critical point $c$ and consider the inverse limit space $X = \UIL$. Since $X$ is \emph{chainable} (\ie it admits an $\eps$-mapping on the interval $[0,1]$ for every $\eps>0$), we call a point $x \in X$  an {\em endpoint} if for any two subcontinua
$A,B \subset X$ containing $x$, either $A \subset B$ or $B \subset A$.
We denote the sets of folding points and endpoints by $\F$ and $\E$, respectively.
Clearly $\E \subset \F$, and we call the points in $\F \setminus \E$ {\em non-end folding points}. 

The structure of folding points is simple if the critical point $c$ of $T$ has a finite orbit,
and (pre)period $n\in\N$. In this case there are $n$ folding points, which are endpoints if $c$ is periodic 
and non-end folding points if $c$ is strictly preperiodic, see \cite{BaMa}.
Every other point in $X$ has a Cantor set of (open) arcs as an (open) neighbourhood.
The structure is more complicated if $c$ has infinite orbit.
Let us call the maximal closed connected sets $A \subset X$ {\em basic arcs} if $\pi_0:A \to I$
is injective, where $\pi_0(x) = x_0$ is the projection on the zero-th coordinate
of $x \in X$. The inverse limit space is the union of its basic arcs, glued together
in an intricate way.
The purpose of this paper is to study and classify the properties of $\F$, $\E$ and $\F \setminus \E$
in terms of whether they are non-empty, (un)countable, or compact sets. We make a subdivision
of $\E$ into {\em flat endpoints} (those that are endpoints of non-degenerate basic arcs),
{\em spiral endpoints} (those that are endpoints of an arc in $X$, but lie in a degenerate basic arc)
and {\em nasty } endpoints (in  \cite{BdCH} called {\em solitary}), \ie those that do not belong to any non-degenerate arc in $X$.
The sets $\F \setminus \E$, $\E_F$, $\E_S$ and $\E_N$ are all preserved under the shift-homeomorphism.

Inverse limit spaces of unimodal maps are also interesting on their own. For instance, they were recently used in the series of papers by Boyland, de Carvalho and Hall \cite{BdCH1,BdCHInvent,BdCH2,3G} in order to give new examples of attractors in surface dynamics. This underlines the fact that the fine topological structure of these inverse limits needs to be better understood. A step in that direction was given recently in  \cite{BdCH} where the authors proved that the topological 
structure of tent inverse limit spaces substantially differs depending on whether the critical orbit 
of the underlying map is dense in the core or not in the following way: if the critical orbit is not dense in the core, then the set $X'\setminus\F$ contains a dense $G_{\delta}$ set and thus a typical point has a Cantor set of arcs neighbourhood.
On the other hand, if the critical orbit is dense in the core, then the set $\E$ contains a dense $G_{\delta}$ set and thus a typical point is an endpoint.
However, it is yet to be determined which one of the sets $\E_N$, $\E_F$ and $\E_S$ is topologically prevalent in $X$, see \cite{BdCH}.

The aim of the paper is to describe these types of folding points and their prevalence
in detail. We give conditions on when $\F \setminus \E$, $\E_F$, $\E_S$ and $\E_N$
are non-empty or equal to $\F$; if these sets are non-empty, we show in Proposition~\ref{prop:denseNonendFoldingpts}, Proposition~\ref{prop:uncountablyEndpts} and Corollary~\ref{cor:uncountablyEndpts} that they are all dense in $\F$. 
The set of endpoints, if infinite, is uncountable (see Proposition~\ref{prop:uncountablyEndpts}), whereas the 
non-end folding points can form a finite, countable or uncountable set, see \cite{GoKnRa}.
If the critical orbit is dense, then $X' = \F$ (see \cite{GoKnRa}),  where $X':=\underleftarrow\lim([T^2(c),T(c)],T)$.
Furthermore, Theorem~\ref{thm:FisE} gives a characterisation of those unimodal maps where $\mathcal{F}=\mathcal{E}$ using the concept of persistent recurrence, which extends results from \cite{Al} to full generality. 
Several other open questions are answered throughout the paper as well. 

 One of the basic questions that can be further investigated is whether there exists a general characterisation (of the parameters) of the H\'enon, Lozi, and Poincar\'e sections of Lorenz attractors for which there exist endpoints, or for which all the folding points are endpoints. Additionally, it is well known that for interesting parameters, these attractors are not locally connected. Hence it may happen that there exist endpoints which are not included in any non-degenerate arc of the attractor (a prominent example of such a continuum is the pseudo-arc, where every point possesses that property). 
Here we show the existence of such points in unimodal inverse limit spaces, and in particular they exist in inverse limits of infinitely renormalisable maps (cf. Theorem~\ref{thm:nastyIR}) and the self-similar inverse limits constructed from the dense $G_\delta$ set $\cA$ of parameters defined in \cite{BBD} (cf. Corollary~\ref{cor:nasty_pts_exist}). 
With further developments and adaptations, the techniques in this paper might be adapted to investigate the inhomogeneities of the mentioned attractors as well.

This paper is organised as follows. In Section~\ref{sec:prel} we give some preliminary definitions and notation
on symbolic dynamics and Hofbauer towers (which are besides standard topological methods the main two techniques used throughout the paper) and inverse limit spaces. Then, in Section~\ref{sec:subcon} we give preliminaries on subcontinua of unimodal inverse limit spaces.
Section~\ref{sec:endpoint} deals with general properties of folding points and with general properties of endpoints. Furthermore,
properties of flat, spiral and nasty endpoints are studied, respectively.
Finally, Appendix~\ref{Appendix} provides a positive answer to Question 6.4.8 of \cite{BB} and demonstrates that the dense $G_\delta$ set of parameters $\cA$ of Theorem 4 in \cite{BBD} corresponds precisely to the collection of tent maps with $\overline{\{T^{S_k}(c)\} }=[T^2(c),T(c)]$, which is a step towards improving the results from~\cite{BBD}.

\section{Preliminaries and notation}\label{sec:prel}
By definition, a \emph{continuum} is a nonempty compact connected metric space.
We will work with two families of unimodal maps on $I := [0,1]$;
the \emph{family of tent maps} $T_s(x) := \min\{ sx , s(1-x)\}$, $s \in (1,2]$, $x\in I$ and for some results also
the \emph{logistic family} $Q_a(x) := ax(1-x)$, $a \in [3,4]$ and $x\in I$. The latter family includes
infinitely renormalisable maps, i.e., there are nested cycles $J_i\subset I$ of periodic intervals of
period $p_i$ (where $p_{i+1}$ is a multiple of $p_i$) and the critical omega-limit 
set $\omega(c) = \cap_i J_i$ is a Cantor set. Such maps give rise to sequences of
nested non-trivial subcontinua and a Cantor set of nasty endpoints, see
Subsection~\ref{sec:nasty}.
We will use $T$ to denote the tent map (so $T = T_s$) unless otherwise stated when $T=Q_a$.
In either of the cases, the point $0$ is fixed and the \emph{critical point} is always $c = \frac12$. Write $c_k := T^k(c)$.
With our choice of parameters, $c_2 < c < c_1$. The interval $[c_2, c_1]$, called the {\em core}, 
is $T$-invariant, and contains a fixed point $r \neq 0$ in its interior. 

Let the \emph{inverse limit space}
$$
X :=\UIL=\{(\ldots, x_{-2}, x_{-1}, x_0): T(x_{-i})=x_{-(i-1)}, i\in\N\}
$$
be the collection of all backward orbits, equipped with the \emph{product metric}
$d(x,y) := \sum_{i \leq 0} 2^i |x_i-y_i|$. 
Denote by $\pi_{i}: X \to I$, $\pi_{i}(x) := x_{-i}$, 
the \emph{coordinate projections} for $i\in\N_0 := \N \cup \{ 0 \}$.
The \emph{shift homeomorphism} $\sigma:X \to X$ is defined by
$$
\sigma(\ldots, x_{-2}, x_{-1}, x_0) := (\ldots, x_{-2}, x_{-1}, x_0, T(x_0)).
$$
We can restrict $T$ to the core; this {\em core inverse limit space} $\CUIL$
will be denoted by $X'$. It is well-known that $X$ is the disjoint union of the core
inverse limit space $X'$ and a ray with endpoint $(\ldots 0,0,0)$,
and also that $X'$ is indecomposable if $T_s$ has slope $\sqrt{2} < s \leq 2$.

Since the situation regarding folding points when $\orb(c)$ is finite is straightforward
(as described in the introduction), we will assume from now on that $c$ has an infinite orbit.

For a set $A\subset\R$ we denote its closure, boundary and interior in $\R$ by $\overline{A}$, $\partial A$ and $A^\circ$,
 respectively.

\subsection{Cutting times}
We recall some notation from Hofbauer towers and kneading maps that we will use later in the paper, for more information on
that topic, see \eg \cite{BB}.
Define inductively $D_1 := [c, c_1]$, and 
$$
D_{n+1} := \begin{cases}
           [c_{n+1}, c_1] & \text{ if } c \in D_n;\\
           T(D_n) & \text{ if } c \notin D_n.
          \end{cases}
$$
We say that $n$ is a {\em cutting time} if $c \in D_n$, and we number them by
$S_0, S_1,  \ldots = 1, 2, \ldots$
The difference between consecutive cutting times is again a cutting time,
so we can define the kneading map $Q:\N \to \N_0$ as
$$
S_{Q(k)}:=S_k - S_{k-1}.
$$
Furthermore, we can check by induction that $D_n = [c_n, c_{\beta(n)}]$ (or $=[c_{\beta(n)}, c_n]$)
where $\beta(n) = n-\max\{ S_k : S_k < n\}$.
For every $k\in \N_0$ let $z_k\in [c_2,c)$ and $\hat z_k:= 1-z_k\in (c,c_1]$ be the {\em closest precritical points}, i.e.,
$T^{S_k}(z_k) = T^{S_k}(\hat z_k) = c$ and $T^j([z_k, \hat z_k]) \not\owns c$ for $0 < j < S_k$.

We establish Equation~(\ref{eq:zzz}) which we will use as a tool in several later places  in the paper. 
Let $\kappa := \min\{ i \geq 2 : c_i > c\}$ (which is finite because $s < 2$).
Define
\begin{equation}\label{eq:upsilon}
\Upsilon_k := [z_{k-1}, z_k) \cup  (\hat z_k, \hat z_{k-1}],
\end{equation}
for $k\in \N_0$. 
Here we set $\hat z_{-1} = c_1$ and $z_{-1} = c_2$. If $\kappa = 3$, then
$z_0 < c_2$, and in this case we define $z_0=c_2$.

Since $c$ is not periodic, $z_n \neq c_k \neq \hat z_n$ for all $n, k \geq 1$.
We argue that $c_{S_k}\in \partial\Upsilon_{Q(k+1)}$ for $k=0,1$, and
\begin{equation}\label{eq:zzz}
c_{S_k}\in \Upsilon_{Q(k+1)}^\circ = (z_{Q(k+1)-1}, z_{Q(k+1)}) \cup (\hat z_{Q(k+1)}, \hat z_{Q(k+1)-1})
\end{equation}
for $k\geq 2$. Without loss of generality, let us assume that $c_{S_k} < c$.
Let $n\in\N$ be minimal
such that $z_n\in(c_{S_k}, c)$. Note that $T^{S_{Q(k+1)}}((c_{S_k}, c))=(c_{S_{Q(k+1)}}, c_{S_{k+1}})\ni c$ 
for every $k\in \N_0$
and by the choice of $n$ it follows that $T^{S_{Q(k+1)}}(z_n)=c$, thus $n=Q(k+1)$ (note that the last statement 
does not hold for $n=-1,0$ which leads to a different conclusion as in (\ref{eq:zzz}) for the two cases).


\begin{definition}\label{def:longBranched}
If $J$ is a maximal interval of monotonicity of $T^k$, then $T^k(J)$ is called a \emph{branch} of $T^k$. 
It follows by induction that every branch of $T^k$ is equal to $D_n$ for some $n \leq k$.
We say that $T$ is \emph{long-branched} if
$\inf_n |D_n| > 0$ (or equivalently, the kneading map is bounded, see \cite[Proposition 6.2.6]{BB}).
\end{definition}

Note that $T$ is long-branched if $c$ is non-recurrent, but there are also long-branched maps with
recurrent critical points, see e.g.\ \cite{Br0}.

\subsection{Symbolic dynamics}

The symbolic itinerary of the critical value $c_1 \in [0,1]$ under the action of $T$ is 
called the {\em kneading sequence}, and we denote it as $\nu = \nu_1\nu_1\nu_3\dots$,
where $\nu_i = 0$ if $c_i < c$ and $\nu_i = 1$ if $c_i > c$.
Analogously, to each $x \in \UIL$, we can assign a symbolic sequence
$\overline{x} =\ovl{x}.\ovr{x} = \ldots s_{-2}s_{-1}.s_0s_1\ldots \in \{ 0, \frac 01, 1 \}^{\Z}$ where 
$$
s_{-i} = \begin{cases}
                                                 0 & \pi_i(x) < c,\\
                                                 \frac 01 & \pi_i(x) = c,\\
                                                 1 & \pi_i(x) > c,
                                                \end{cases}
\hspace{30pt} s_{i} = \begin{cases}
                                                0 & T^i(x) < c,\\
                                                \frac 01 & T^i(x) = c,\\
                                                1 & T^i(x) > c,
                                                \end{cases}
\hspace{10pt} i \geq 0.
$$         
Here $\frac 01$ means that both $0$ and $1$ are assigned to $x$. Since we assumed that $c$ has an infinite orbit, 
this can happen only once, i.e., every point has at most two symbolic itineraries.

For a fixed left-infinite sequence $\ovl s=\ldots s_{-2}s_{-1}\in\{0, 1\}^{\N}$, the subset
$$
A(\ovl s):= \overline{\{ x \in X : \ovl{s}\in\overleftarrow{x}\}} 
$$
of $X$ is called a {\em basic arc}. As mentioned in the introduction, 
$A(\overleftarrow{x})$ is the maximal closed arc $A$ containing $x$ such that
$\pi_0:A \to I$ is injective.
In \cite[Lemma 1]{Br1} it was observed that $A(\overleftarrow{x})$ is indeed an arc or degenerate (i.e., a single point).

\section{Subcontinua}\label{sec:subcon}

In this section we describe some general properties of subcontinua of $X$, taking \cite{BrBr} as a starting point.
If $H$ is a subcontinuum of $X$, then the continuity of the projections guarantee that
$\pi_{i}(H)$ are intervals for every $i\in\N_0$. 
Furthermore, if there is $k\in\N$ so that $c \notin \pi_{i}(H)$ for all $i > k$, 
then $H$ is either a point or an arc
(because then we can parametrise $H$ by $t \in \pi_k(H)$).
Furthermore, when $T$ is locally eventually onto the core (and this is true
for $T = T_s$ for all $s \in (\sqrt{2},2]$), then 
$H$ is a proper subcontinuum of $X$ if and only if $|\pi_{i}(H)| \to 0$ as $i\rightarrow \infty$.  
As a consequence (see Proposition~\ref{prop:longbranched} and Proposition 3 in \cite{BrBr}),
if $T$ is long-branched, then the only proper subcontinua of $X'$ are arcs.

Let $H\subset X$ be a proper subcontinuum and let $\{n_i\}_{i\in\N}\subset \N_0$ be 
its \emph{critical projections}; \ie $c\in\pi_{n}(H)$ if and only if $n\in \{n_i\}_{i\in\N}$. 
Since $H$ and $\sigma^{n_1}(H)$ are homeomorphic, we can assume for our purposes that $n_1 = 0$.

\begin{definition}\label{def:critical_projections}
 For $i\geq 1$ let $M_{n_i}$ denote the closure of component of $\pi_{n_i}(H)\setminus \{c\}$ such 
 that $T^{n_i-n_{i-1}}(M_{n_i})=\pi_{n_{i-1}}(H)$. 
 Denote by $L_{n_i}$ the closure of the other component of $\pi_{n_i}(H)\setminus \{c\}$. 
 If both $\pi_{n_i}(L_i) = \pi_{n_i}(M_i) = \pi_{n_{i-1}}(H)$, then denote by $M_{n_i}$ the component 
 that contains the point $T^{n_{i+1}-n_i}(c)$ as a boundary point. 
\end{definition}

\begin{proposition}[Proposition 1 in \cite{BrBr}]\label{prop:denseray}
Any subcontinuum $H\subset X$ is either a point or it contains a dense line.
\end{proposition}	

A specific case of Proposition~\ref{prop:denseray} is when we take $H=X'$. Since in this case we can take 
for $M_n=[c,c_1]$, it follows from the proof of Proposition~\ref{prop:denseray} that
the arc-component $\cR$ of the fixed point $\rho=(\ldots,r,r)$ of $T$ is the required dense line. Recall that an {\em arc-component} of a point $x\in X$ is the union of all arcs in $X$ which contain $x$.
 
Therefore, we obtain the following corollary:

\begin{corollary}\label{cor:Rdense}
The arc-component $\cR$ of $\rho$ is a dense line in $X'$.
\end{corollary}

A {\em composant} $\V_x$ of $x\in X$ is a union of all proper subcontinua of $X$ containing $x$.
An indecomposable continuum consists of uncountably many pairwise disjoint dense
composants, see \cite{Na}. If $X'$ is indecomposable, then $\cR = \V_\rho$, and this gives a 
negative answer to a question of Raines \cite[Problem 5]{In}, whether in every $X'$ such that $\omega(c)=[c_2,c_1]$, 
every composant contains homeomorphic copies of every tent inverse limit spaces.

Some further, more general properties of the $M_{n_i}$ and $L_{n_i}$ allow a description of subcontinua $H$.
Brucks \& Bruin \cite{BrBr} observed that the topologist's {\em $\sin(1/x)$-continuum}
can appear as a subcontinuum of $X'$, see Figure~\ref{fig:sin_curve}. 
This is any space homeomorphic to the graph of the function $\sin \frac1x$, $x \in (0,1]$ in $\R^2$, together with the arc 
$A = \{0\} \times [-1,1]$ that the graph compactifies on.
An {\em arc+ray continuum} (called {\em Elsa continuum} in \cite{Na-Elsa})
is a generalisation of $\sin(1/x)$-continua: it is any continuum consisting of an arc and 
a ray compactifying on it.
A {\em double $\sin(1/x)$-continuum} is any space homeomorphic to the graph of the function $\sin( \frac1{x(1-x)})$, $x \in (0,1]$ in $\R^2$, 
together with the two arcs 
$\{ 0\} \times [-1,1]$ and $\{ 1 \} \times [-1,1]$ that the graph compactifies on. 
In the same way, we can define {\em double arc+ray continua}.

\begin{figure}[!ht]
	\centering
	\begin{tikzpicture}[scale=2.5]
	\draw[thick] (0,0.34)--(2,0.34);
	\draw[thick]  (0,0.5)--(2,0.5);
	\draw[thick]  (0,0.56)--(2,0.56);
	\draw[thick]  (0,0.7)--(2,0.7);
	\draw[thick]  (0,0.78)--(2,0.78);
	\draw[thick]  (0,0.9)--(2,0.9);
	\draw[thick,domain=90:270] plot ({0.08*cos(\x)}, {0.42+0.08*sin(\x)});
	\draw[thick,domain=90:270] plot ({0.07*cos(\x)}, {0.63+0.07*sin(\x)});
	\draw[thick,domain=90:270] plot ({0.06*cos(\x)}, {0.84+0.06*sin(\x)});
	\draw[thick,domain=270:450] plot ({2+0.03*cos(\x)}, {0.53+0.03*sin(\x)});
	\draw[thick,domain=270:450] plot ({2+0.04*cos(\x)}, {0.74+0.04*sin(\x)});
	\draw[thick,domain=270:450] plot ({2+0.02*cos(\x)}, {0.92+0.02*sin(\x)});
	\thicklines
	\draw[thick] (0, 1)--(2,1);
	\node[circle,fill, inner sep=0.05] at (1,0.925){};
	\node[circle,fill, inner sep=0.05] at (1,0.95){};
	\node[circle,fill, inner sep=0.05] at (1,0.975){};
	\draw[thick] (3,0.34)--(5,0.34);
	\draw[thick]  (3,0.5)--(5,0.5);
	\draw[thick]  (3.5,0.56)--(5,0.56);
	\draw[thick]  (3.5,0.7)--(4.5,0.7);
	\draw[thick]  (3,0.78)--(4.5,0.78);
	\draw[thick]  (3,0.9)--(5,0.9);
	\draw[thick,domain=90:270] plot ({3+0.08*cos(\x)}, {0.42+0.08*sin(\x)});
	\draw[thick,domain=90:270] plot ({3.5+0.07*cos(\x)}, {0.63+0.07*sin(\x)});
	\draw[thick,domain=90:270] plot ({3+0.06*cos(\x)}, {0.84+0.06*sin(\x)});
	\draw[thick,domain=270:450] plot ({5+0.03*cos(\x)}, {0.53+0.03*sin(\x)});
	\draw[thick,domain=270:450] plot ({4.5+0.04*cos(\x)}, {0.74+0.04*sin(\x)});
	\draw[thick,domain=270:450] plot ({5+0.02*cos(\x)}, {0.92+0.02*sin(\x)});
	\thicklines
	\draw[thick] (3, 1)--(5,1);
	\node[circle,fill, inner sep=0.05] at (4,0.925){};
	\node[circle,fill, inner sep=0.05] at (4,0.95){};
	\node[circle,fill, inner sep=0.05] at (4,0.975){};
	\end{tikzpicture}
	\caption{A $\sin(1/x)$-continuum and a more general arc+ray continuum.} 
	\label{fig:sin_curve}
\end{figure}
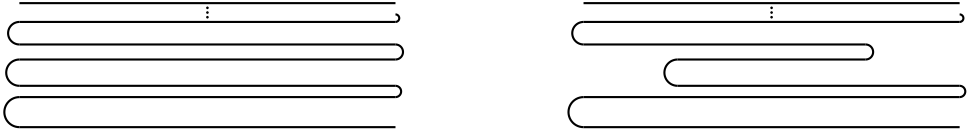

More general is the following result.

\begin{proposition}[Theorem 1 \cite{BrBr}]\label{prop:H}
Let $H$ be a subcontinuum with critical projections $\{ n_i \}$ and $\pi_{n_i}(H) = M_{n_i} \cup L_{n_i}$ as above.
Then if $c \notin T^{n_i-n_{i-1}}(L_{n_i})$ for all sufficiently large $i$, 
then $H$ is a point, an arc, a $\sin \frac1x$-continuum, or a double $\sin(1/x)$-continuum. 
\end{proposition}

We call the arc in this proposition a {\em direct spiral} (a countable infinite union of 
basic arcs whose diameters uniformly decrease to $0$ together with a spiral endpoint) and the 
$\sin(1/x)$-continuum in part (1) a {\em basic $\sin(1/x)$-continuum}
(because its bar is the finite union of a basic arcs).\footnote{
From \cite[Theorem 1.2]{Br3}, it follows that there are core inverse limit spaces with non-basic $\sin(1/x)$-continua,
possibly infinitely many of them.}
Proposition~\ref{prop:spiral} below gives conditions under which such subcontinua exist.

\begin{proposition}\label{prop:spiral}	
Assume that there is an increasing subsequence $(k_i)_{i \geq 0}\subset \N$ such that 
\begin{equation}\label{eq:Qcond}
Q(k_i) = k_{i-1} \text{ and } Q(Q(k_i-1) +1) < k_{i-1}-1 \text{ for all } i \geq 1.
\end{equation}
Then $X'$ contains a subcontinuum which is 
direct spiral if $\lim_i Q(k_i+1) = \infty$ and a basic $\sin(1/x)$-continuum 
if $\liminf_i Q(k_i+1) < \infty$. 
\end{proposition}

\begin{proof}
We create a subcontinuum $H$ with critical projections $n_i = S_{k_i}$, so that
$n_i-n_{i-1} = S_{k_i} - S_{k_{i-1}} = S_{k_i} - S_{Q(k_i)} = S_{k_i-1}$ for all $i \geq 1$. 
The projections $H_{n_i}:=\pi_{n_i}(H)$ are constructed such that $L_{n_{i-1}} = [c,c_{S_{k_i-1}}]$ for all $i \geq 0$ and
 $M_{n_i} = [c,a_i]$ where the points $a_i$ are chosen such that $T^{n_i-n_{i-1}}$ maps
$[c,a_i]$ monotonically onto $[c_{S_{k_i-1}}, a_{i-1}]$. 

First we inductively show that we can always find such $a_i\in \Upsilon^{\circ}_{k_i-1}=(z_{k_i-2}, z_{k_i-1})\cup (\hat z_{k_i-1}, \hat z_{k_i-2})$ for every $i\in\N$. We can take \eg $k_0=3$ and $a_1\in (z_{k_0-2},z_{k_0-1})$. Assume such $a_{i-1}$ has been constructed for some $i\in\N$ and assume without the loss of generality that $a_{i-1}\in (z_{k_{i-1}-2}, z_{k_{i-1}-1})$. Note that $T^{n_i-n_{i-1}}(\Upsilon^{\circ}_{k_{i}-1})=T^{S_{k_i-1}}(\Upsilon^{\circ}_{k_{i}-1})=(c_{S_{Q(k_{i}-1)}}, c)$. 
Since $Q( Q(k_{i}-1)+1) < k_{i-1}-1$,  it follows by equation~\eqref{eq:zzz} 
that $c_{S_{Q(k_{i}-1)}}\notin (z_{k_{i-1}-2}, \hat z_{k_{i-1}-2})$, so 
$T^{S_{k_i-1}}(\Upsilon^{\circ}_{k_i-1})\supset (z_{k_{i-1}-2}, z_{k_{i-1}-1})\ni a_{i-1}$ 
and we can easily choose $a_i$. Note that if $a_{i-1}\in (\hat z_{k_{i-1}-1}, \hat z_{k_{i-1}-2})$, 
then we would have $T^{S_{k_i-1}}(\Upsilon^{\circ}_{k_i-1})\supset (\hat z_{k_{i-1}-1}, \hat z_{k_{i-1}-2})\ni a_{i-1}$.

Furthermore, by \eqref{eq:zzz},  $Q(k_{i+1})$ is the smallest $n\in\N$ such that $c_{S_{k_{i+1}-1}}\notin (z_n, \hat z_n)$. 
We assumed that $Q(k_{i+1})=k_i$, so it follows that 
$n_i = S_{k_i} = \min\{ n \geq 1 : c \in T^n(L_{n_i})\}$, and $T^{n_i}(L_{n_i}) = D_{S_{k_{i+1}}}$.

The intervals $D_{S_{k_i}}$ form a nested sequence of neighbourhoods of
$c$, converging to a point if and only if $|D_{S_{k_i}}| \to 0$, which is the case if
and only if $\lim_i Q(k_i+1) = \infty$, see \eqref{eq:zzz}.
In this case, $H$ is a direct spiral, i.e., Proposition~\ref{prop:H} applies. Otherwise, it is a 
basic $\sin(1/x)$-continuum,
with $\cap_i D_{S_{k_i}}$ equal to the $0$-th projection of the bar of $\sin(1/x)$-continuum.
\end{proof}

\begin{remark}\label{rem:Qcond2}
	If \eqref{eq:Qcond} is relaxed to:
	\begin{equation}\label{eq:Qcond2}
	Q(Q(k_i-1) +1) < k_{i-1}-1 < Q(k_i) \text{ for all sufficiently large } i,
	\end{equation}
	then we can proceed similarly as in the previous proposition. In this case we construct a subcontinuum $H$ with critical projections
	$n_1 = 0$, $n_i = \sum_{j=1}^{i-1} S_{k_j-1}$ for $i \geq 2$ (so $n_i - n_{i-1} = S_{k_i-1}$) and
	$H_{n_i} = [c_{S_{k_{i+1}-1}}, a_i]$ where $(a_i)_{i \geq 1} \subset [c_2, c_1]$ is a sequence of points such that
	$T^{S_{k_i-1}}: [c,a_i] \to [c_{S_{k_i-1}}, a_{i-1}]$ is monotone.
	That is, $L_{n_i} = [c, c_{S_{k_{i+1}-1}}]$, and because
	$Q(Q(k_{i+1}-1)+1) < k_i-1$, it is indeed possible to choose $a_i$ such that  $M_{n_i} := [c, a_i] \owns z_{k_i-1}$,
	whence $c \in T^{n_i-n_{i-1}}(M_{n_i})$.
	Because $Q(k_{i+1}) > k_i-1$, we have $c \notin T^{n_i-n_{i-1}}(L_{n_i}) \subset L_{n_{i-1}}$.
	Therefore $H$ is a direct spiral or a basic $\sin \frac1x$-continuum, depending on whether $\cap_i T^{n_i}(L_{n_i})$
	is a point or an arc.
\end{remark}

\section{Endpoints and non-end folding points}\label{sec:endpoint}

\subsection{Folding points}\label{sec:folding}
First we address folding points in general.
The following characterisation of folding points is due to Raines.

\begin{proposition}[Theorem 2.2 in \cite{Raines}]\label{prop:foldpts}
	A point $x\in X'$ is a folding point if and only if $\pi_n(x)$ belongs to the omega-limit set $\omega(c)$ 
	for every $n\in\N$.
\end{proposition}

Note that $\F=\underleftarrow{\lim}\{\omega(c), T|_{\omega(c)}\}$. 
Since $\omega(c)$ is compact, the set $\F \cap X'$ is also compact and non-empty.
It also follows that $\F = X'$ if $\omega(c) = [c_2, c_1]$ and $\F$ is nowhere
dense if $\omega(c)$ is nowhere dense. So excluding renormalisable maps,
$\F \cap X'$ is either nowhere dense or equal to $X'$.
Furthermore, if $c$ is recurrent, then $\omega(c) = \overline{\orb(c)}$ is either finite or
perfect and therefore uncountable. In the latter case $\F$ is also uncountable.

\begin{lemma}[Proposition 2 in \cite{Br1}]\label{lem:non-recurrent}
	If $\orb(c)$ is infinite and $c$ is non-recurrent, then there exist infinitely many 
	folding points and no endpoints.
\end{lemma}

It is possible that $X$ has countably infinitely many non-end folding points.
An example is the map with \emph{kneading sequence} (i.e., symbolic itinerary of the critical value $c_1$)
$$
\nu = 1.0.0.11.0.11.11.0.11.11.11.0.11.11.11.11.0.11.11.11.11.11 \dots
$$
Indeed, if the critical point is non-recurrent, by Proposition~\ref{prop:foldpts}
the folding points have two-sided itinerary $\dots 1111 \dots$ or
$\dots 111101111 \dots$, and  $\F \cap X'$ has only isolated points, except for the fixed point $\rho$. Furthermore, 
$\F$ is countably infinite.
In general,
it was shown in \cite[Theorem~6.2]{GoKnRa} that 
if $\omega(c)$ is homeomorphic to $n\in\N$ disjoint copies of $S_0=\{0\}\cup\{1/k: k\in\N\}$
(for which $c$ must be non-recurrent), 
then the set of folding points is homeomorphic to $n$ copies of $S_0$ too and is thus countable. 
The example given with $\nu$ above corresponds to the $n=1$ case. Note that the accumulation point in the 
example is fixed under the shift $\sigma$. Analogously, the $n$ accumulation points in the general case 
will be periodic under $\sigma$.

Clearly the number of folding points is uncountable if $\omega(c)$ is uncountable, but 
also when $\omega(c)$ is countable, it can happen that
the set of folding points is uncountable.
This is shown in \cite{GoKnRa}, together with more interesting results on the number of folding points in $X$. 
However, the following natural problem that arises is to our knowledge still open.

\begin{problem} 
	Give necessary and sufficient conditions on $c$ so that the corresponding
	inverse limit space $X$ has countably infinitely many folding points. 
\end{problem}

Note that $X$ contains endpoints if and only if $T$ is recurrent. Moreover, if $X$ contains infinitely many endpoints, then it contains uncountably many of them, see Proposition~\ref{prop:uncountablyEndpts}. Thus $X$ which contain only countably infinitely many folding points must be generated by $T$ with non-recurrent critical orbit, and then every folding point will be a non-end folding point. So the problem above can be asked in more general form:

\begin{problem}
	Give necessary and sufficient conditions on $c$ so that the corresponding
	inverse limit space $X$ has countably infinitely many {\it non-end} folding points. 
\end{problem}

\subsection{Endpoints}\label{sec:end}
Now we focus on endpoints of $X$.
We give a symbolic classification of endpoints, based on \cite[Section 2]{Br1}.
For every basic arc $A(\overleftarrow{x})$, where $\ovl x=\ldots s_{-2}s_{-1}\in\{0,1\}^{\N}$, we define 
\begin{eqnarray*}
	\NL(\overleftarrow{x})
	&:=& \{n>1 : s_{-(n-1)}\ldots s_{-1}=\nu_1\nu_2\ldots \nu_{n-1}, \#_1(\nu_1\ldots \nu_{n-1})\textrm{ odd}\}, \\
	\NR(\overleftarrow{x})
	&:=&\{n\geq 1 : s_{-(n-1)}\ldots s_{-1}=\nu_1\nu_2\ldots \nu_{n-1}, \#_1(\nu_1\ldots \nu_{n-1})\textrm{ even}\}.
\end{eqnarray*}
and
$$
\tau_L(\overleftarrow{x}) := \sup\NL(\overleftarrow{x}) \quad \text{ and } 
\quad \tau_R(\overleftarrow{x}) :=\sup\NR(\overleftarrow{x}).
$$
The quantities $\tau_L$ and $\tau_R$  first appeared in \cite{Br1} in order to study the number of 
endpoints of unimodal inverse limit spaces $X$. 
In the definition of $\NR(\overleftarrow{x})$ we allow $n=1$, so since $\#_1(\emptyset)$ is even, it follows immediately that $N_R(\ovl x)$ is a non-empty set. On the other hand, supremum is well-defined  for $\tau_L$ as well. Namely, if $s_{-1}=1$, then $s_{-1}=\nu_1$ and $\#_1(\nu_1)$ is odd. In case $s_{-1}=0$ we find the smallest $n>2$ so that $s_{-(n-1)}=1$, which indeed exists since $\ovl x\neq \ovl 0$. Thus $N_L(\ovl x)$ is non-empty also. 

Now we restate some lemmas from \cite{Br1} in our setting.

\begin{lemma}\label{lem:first}(\cite{Br1}, Lemma 2 and Lemma 3)
	If $A(\overleftarrow{x}) \in \{ 0,1\}^{\N}$ is such that both
	$\tau_L(\overleftarrow{x}),\\
	 \tau_R(\overleftarrow{x})
	<\infty$, then 
	$$
	\pi_0(A(\overleftarrow{x}))=[T^{\tau_L(\overleftarrow{x})}(c), T^{\tau_R(\overleftarrow{x})}(c)] = D_n,
	$$
for $n = \mathrm{max}\{ \tau_L(\overleftarrow{x}), \tau_R(\overleftarrow{x}) \}$.
Without the restriction that $\tau_L(\overleftarrow{x}), \tau_R(\overleftarrow{x})<\infty$, we have
\begin{eqnarray*}
\sup\pi_0(A(\overleftarrow{x}))=\inf\{c_n:  n \in \NR(\overleftarrow{x}) \},\\
\inf\pi_0(A(\overleftarrow{x}))=\sup\{c_n:  n  \in \NL(\overleftarrow{x}) \}.
\end{eqnarray*}
\end{lemma}

This gives the following symbolic characterisation of endpoints. 

\begin{proposition}\label{prop:endpt}\cite[Proposition 2]{Br1}
A point $x\in X$ such that \footnote{Note that in the following proposition we assume that $\pi_i(x)\neq c$ for every $i>0$. 
If there exists $i>0$ such that $\pi_i(x)=c$, then we apply the proposition to $\sigma^i(x)$ 
and use the fact that shift $\sigma$ preserves endpoints.}
 $\pi_i(x)\neq c$ for every $i<0$ is an endpoint of $X$ if and only if 
$\tau_L(\overleftarrow{x})=\infty$ and $\pi_0(x)=\inf\pi_0(A({\ovl{x}}))$ or 
$\tau_R(\overleftarrow{x})=\infty$ and $\pi_0(x)=\sup\pi_0(A({\ovl{x}}))$.
\end{proposition}

The following proposition follows implicitly from the proof of Corollary 2 in \cite{Br1}. It shows that if $c$ is recurrent, then $\#(\E \cap X') = n\in\N$ if and only if $c$ is $n$-periodic, and otherwise
$\#(\E \cap X') = 2^{\aleph_0}$. 
We prove here an extension of that statement.
	
\begin{proposition}\label{prop:uncountablyEndpts}
If $\orb(c)$ is infinite and $c$ is recurrent, then the core inverse limit space $X'$ has 
uncountably many endpoints.
Moreover, $\E$ has no isolated points and is dense in $\F$.
\end{proposition}
	
\begin{proof}
Since $c$ is recurrent, for every $k\in\N$ there exist infinitely many  $n\in\N$ such that 
$\nu_1\ldots \nu_{n}=\nu_1\ldots \nu_{n-k}\nu_1\ldots \nu_k$. 

Take a sequence $(n_j)_{j\in\N}$ such 
that $\nu_1\ldots \nu_{n_{j+1}}=\nu_1\ldots \nu_{n_{j+1}-n_j}\nu_1\ldots \nu_{n_j}$ for every $j\in\N$. 
Then the basic arc given by the itinerary 
$$
\overleftarrow{x}:=\lim_{j\to\infty} \nu_1\ldots \nu_{n_j},
$$
is admissible and $\tau_L(\overleftarrow{x})=\infty$ or $\tau_R(\overleftarrow{x})=\infty$.
Therefore, $A(\ovl x)$ contains an endpoint. Note that, since $\nu$ is not periodic, $\ovl x$ 
is also not periodic and thus $\sigma^k(\ovl x)\neq\ovl x$ for every $k\in\N$. 

To determine the cardinality of endpoints, we claim that 
for every fixed $n\in\N$ there are $m_2>m_1>n$ such that
$$
\nu_1\ldots\nu_{m_2}=\nu_1\ldots\nu_{m_2-n}\nu_{1}\ldots\nu_{n}, \quad 
\nu_{1}\ldots\nu_{m_1}=\nu_1\ldots\nu_{m_1-n}\nu_1\ldots\nu_n, 
$$
but  $\nu_1\ldots\nu_{m_1}$ is not a suffix of $\nu_1 \dots \nu_{m_2}$. 
Indeed, if $m_2$ does not exist, then 
$\ovl x=(\nu_1\ldots\nu_{m_1-n})^{-\infty}\nu_1\ldots\nu_n$ would have an eventually periodic tail, which is a contradiction with $\sigma^k(\ovl x)\neq\ovl x$ for every $k\in\N$.

We conclude that for every $n_j$ there are at least two choices of $n_{j+1}$ such that the corresponding 
tails $\ovl x$ are different, and have $\# (N_L(\ovl x) \cup N_R(\ovl x)) = \infty$. 
It follows that there are uncountably many basic arcs containing at least one endpoint of $X$.

To show that $\E$ contains no isolated points and is in fact dense in $\F$, 
take any folding point $x$ with two-sided itinerary $\ldots s_{-2}s_{-1}.s_0s_1s_2\ldots$ 
Then, for every $k\in\N$, there exists $n\in\N$ such that $s_{-k}\ldots s_{k}=\nu_n\ldots\nu_{n+2k}$. 
Using the arguments as in the previous paragraphs of the proof, we can find a basic arc with itinerary 
$\ovl y=\ldots \nu_1\ldots\nu_{n-1}\nu_n\ldots\nu_{n+2k}$ and such that 
$\tau_L(\ovl y)=\infty$ or $\tau_R(\ovl y)=\infty$. So $\sigma^{-k}(\ovl y)$ contains an endpoint 
with itinerary $\ldots \nu_n\ldots\nu_{n+k}.\nu_{n+k+1}\ldots\nu_{n+2k}\ldots$ 
Since $k\in\N$ was arbitrary, 
we conclude that there are (in fact, uncountably many) endpoints arbitrarily close to the point $x$. 
\end{proof}

For the more detailed properties of endpoints, we make a distinction between
flat $\E_F$, spiral $\E_S$ and nasty endpoints $\E_N$. It is clear from
the definitions (see the introduction) that $\sigma^k(\E_F)=\E_F$, $\sigma^k(\E_N)=\E_N$, and $\sigma^k(\E_S)=\E_S$ 
for every $k\in\Z$. Therefore:

\begin{corollary}\label{cor:uncountablyEndpts}
The orbit $\{\sigma^n(x): n\in\N_0\}$ of every $x\in\E$ is dense in $\F$.
In particular, $\E_S, \E_N, \E_F$ are not closed in $\F$ unless they are empty or equal to $\F$. 
\end{corollary}

\begin{proof}
 Since $\orb(c)$ is a dense subset of $\omega(c)$, for every $y \in \F$, $n \in \N_0$
 and neighbourhood $U \owns \pi_n(y)$, 
 there is $m \in \N_0$ such that $c_m \in U$.
 As in the proof of Proposition~\ref{prop:uncountablyEndpts}, since every $x \in \E$ has a sequence $(n_j)_{j \geq 1}$ such that $\pi_{n_j}(x) \to c$,
 we can find $j \geq 1$ such that $\pi_n(\sigma^k(x)) \in U$ for $k = m-n_j$.
 But this implies that $\{\sigma^k(x) : k \in \Z\}$ is dense in $\F$.
\end{proof}

\begin{remark}
Note that orbits of non-end folding points are not necessarily dense in $\F$. 
For example, assume $\omega(c)=[c_2, c_1]$. Then every point is a folding point, 
and in particular the fixed point $\rho\in\F\setminus\E$. 
\end{remark}

Now that we have more information about endpoints we briefly look back at non-end folding points in order to prove an analogue of Proposition~\ref{prop:uncountablyEndpts} and thus give an insight into the topology of non-end folding points. 
The following proposition holds in particular when $\omega(c)$ is not minimal. Recall that a dynamical system $(Y, f)$ (or just a set $Y$) is called {\em minimal} if it does not contain a non-empty, proper, closed, $f$-invariant subset, or equivalently, if every orbit is dense in $Y$.

\begin{proposition}\label{prop:denseNonendFoldingpts}
	If $\F \setminus \E \neq \emptyset$ in $X$, then $\F\setminus\E$ is dense in $\F$. 
\end{proposition}

\begin{proof}
	Recall that $\F = \underleftarrow\lim(\omega(c),T|_{\omega(c)})$. 
	We can assume that $c$ is recurrent, because otherwise $\F = \F \setminus \E$, and there is nothing to prove.
	Assume that $\F \setminus \E \neq \emptyset$; recall that $\F \setminus \E$ is $\sigma$-invariant.
	Let $\eps > 0$ be arbitrary. We  claim that there is $x \in \F \setminus \E$
	such that $\orb_\sigma(x)$ is $\eps$-dense in $\F$. Since $\eps$ is arbitrary, this will prove the proposition. 
	
	Fix $z = (\dots, z_{-2}, z_{-1}, z_0) \in \F$ such that $z_0 = c$; by recurrence of $c$ such a folding point always exists.
	To prove the claim, find $N = N(\eps) \in \N$ such that $\{ \sigma^n(z)\}_{n=0}^N$ is $\eps/2$-dense in $\F$. 
	There is $\delta>0$ such that
	$\diam(T^j(B_\delta(c))) < \eps/2$ for every
	$0 \leq j \leq N$, so that
	$\{ \sigma^n(x)\}_{n=0}^N$ is $\eps$-dense in $\F$ for every $x = (\ldots,x_{-2},x_{-1},x_0) \in \F$ with $d(x_0, c) \leq \delta$.
	
	If $\omega(c)$ is minimal, then for every $x \in \F \setminus \E$, there is $k \geq 0$ such that 
	$d(T^k(x_0),c) \leq \delta$, so $\{ \sigma^n(x)\}_{n=k}^{N+k}$ is $\eps$-dense in $\F$.
	
	Therefore it remains to verify the non-minimal case, i.e., assume there exists a $T$-invariant closed 
	set $\Omega \subset \omega(c)$ such that $\eta := d(\Omega,c)/3 > 0$.
	If $x\in\underleftarrow{\lim}\{\Omega, T\}$, then $x\in\F \setminus \E$, since $d(x_{-j},c) > \eta$ for all $j \geq 0$.
	Now since $\Omega \subset \omega(c)$, for every $j \in \N$ we can find $k_j \in\N$ such that
	$d(c_{k_j+i}, \Omega) < \eta$ for all $0 \leq i \leq j$. Note that the choice of $\eta$ implies that $d(c_{k_j+i}, c)>\eta$ for all $0 \leq i \leq j$. Since $c$ is recurrent, we can also find
	a minimal $m_j > k_j$ such that $d(c,c_{m_j}) < \delta$.
	Take
	$$
	x^j := \sigma^{m_j}(z) = ( \dots z_{-2}, z_{-1}, c, \dots, c_{k_j}, c_{k_j+1}, \dots, c_{m_j}) \in \F,
	$$
	(recall that $z\in \mathcal{F}$) and let $x$ be any accumulation point of the sequence $(x^j)_{j \geq 1}$.
	Then  $x \in \F$ (because $\F$ is closed), and $d(c,x_{-i}) \geq \delta$ 
	for every $i \geq 1$, due to the minimality of $m_j$. This means that $\tau_L(\ovl x), \tau_R(\ovl x) < \infty$, so $x \notin \E$.
	However, $d(x_0,c) \leq \delta$, so $\{ \sigma^n(x) \}_{n=0}^N$ is $\eps$-dense in $\F$.
	This concludes the proof.
\end{proof}

\subsubsection{Flat endpoints}\label{sec:flat}
Recall that a flat endpoint is an endpoint in $X$ that is an endpoint of its own {\bf non-degenerate} basic arc.
The set of flat endpoints is denoted by $\E_F$.
By Proposition~\ref{prop:endpt},
$\tau_L(\overleftarrow{x})=\infty$ and $\pi_0(x)=\inf\pi_0(A({\ovl{x}}))$ or 
$\tau_R(\overleftarrow{x})=\infty$ and $\pi_0(x)=\sup\pi_0(A({\ovl{x}}))$. Recall that 
$\sup\pi_0(A(\overleftarrow{x}))=\inf\{c_n:  n \in \NR(\ovl{x}) \} >
\inf\pi_0(A(\overleftarrow{x}))=\sup\{c_n:  n  \in \NL(\ovl{x}) \}$.
The following statement is an extension of Proposition 3 from \cite{BrBr}.

\begin{proposition}\label{prop:longbranched}
If the map $T$ is long-branched, then the only proper subcontinua of $X'$ are arcs.
Additionally, $\E \neq \emptyset$ if and only if $c$ is recurrent, and in this case
all endpoints are flat, i.e., $\E = \E_F$.
\end{proposition}

\begin{proof}
	Assume by contradiction that a subcontinuum $H\subset X'$ with critical 
	projections $\{n_i\}_{i\in\N}\subset \N$ is not an arc. By our observations in Section~\ref{sec:subcon},
	the set of critical projections $\{n_i\}_{i\in\N}$ is infinite, and
	there exists $N(i)\in\N$ so that $[c,c_{N(i)}]\subset\pi_{n_i}(H)$. 
	Since $T$ is long-branched, there exists $\delta>0$ so that $|c_{N(i)}-c|>\delta$ for every $i\in\N$. 
	However, this contradicts that $H$ is a proper subcontinuum of $X'$.
	
	The proof of Proposition~\ref{prop:uncountablyEndpts} gives $\E\neq\emptyset$ if and only if $c$ is recurrent. Since there are no arbitrarily short basic arcs, it follows that $\E=\E_F$.
\end{proof}

Thus, if $c$ is recurrent, and $T$ is long-branched, then all endpoints are flat, and Proposition~\ref{prop:uncountablyEndpts} guarantees there are 
uncountably many of them. It is possible that there are additional non-end folding points, cf.\ Theorem~\ref{thm:FisE}.
For instance, if the orientation reversing fixed point $r$ belongs to $\omega(c)$, then
$\rho := (\dots, r,r,r)$ is a non-end folding point.

\begin{problem} Give necessary and sufficient conditions on $T$ such that $\E = \E_F$.\end{problem}

 Since the class of self-homeomorphisms of unimodal inverse limit spaces is so rigid
(all homeomorphisms are isotopic to powers of $\sigma$), no self-homeomorphism $h:X \to X$
can map a flat endpoint to a spiral endpoint (or nasty endpoint). However, we can ask the following.

\begin{problem}
	 Is it possible that a flat endpoint has a neighbourhood that is 
	homeomorphic to a neighbourhood of a spiral endpoint? 
In other words, given unimodal inverse limit $X$, a flat endpoint $x\in X$, and a spiral endpoint $y\in X$, is there a way to
 distinguish $x$ from $y$ locally? Similarly, given two different unimodal inverse limits $X$ and $Y$, a flat endpoint $x\in X$, and a spiral endpoint $y\in Y$, is there a way to distinguish $x$ from $y$ locally?

\end{problem}

\subsubsection{Spiral endpoints}\label{sec:spiral}

Let $\E_S$ denote the \emph{set of spiral endpoints}, i.e., the endpoints that have a {\bf degenerate} basic arc 
but are contained in a non-degenerate arc-component of $X$.
The notion of persistent recurrence was first introduced in \cite{Ly} in connection with the existence 
of wild attractors of unimodal interval maps. It turns out to be the crucial notion for 
the classification of core inverse limits $X'$ for which $\F=\E$. 

\begin{definition}\label{def:pullback}
Let $x=(\ldots, x_{-1}, x_0)\in X$ and let $J\subset I$ be an interval. 
The sequence $(J_n)_{n\in\N_0}$ of intervals is called a \emph{pull-back} of $J$ along 
$x$ if $J=J_0$, $x_{-k}\in J_k$ and $J_{k+1}$ is the largest interval such that $T(J_{k+1})\subset J_k$ 
for all $k\in\N_0$. A pull-back is {\em monotone} if $c\not\in J_n^\circ$ 
for every $n\in\N$. 
\end{definition}

Lyubich \cite{Ly} gave the following definition in the case when $c$ is recurrent:

\begin{definition}\label{def:persrec}
The critical point $c$ is {\em reluctantly recurrent} if there is $\eps>0$ and an arbitrary long (but finite!) 
backward orbit $\bar{y}=(y_{-l}, \ldots, y_{-1}, y)$ in $\omega(c)$ such that the $\eps$-neighbourhood 
of $y\in I$ has monotone pull-back along $\bar{y}$. Otherwise,  $c$ is {\em persistently recurrent}.
\end{definition}

The following lemma shows that one can replace arbitrarily long pull-backs by infinitely long
pull-backs, and this allows us to interpret reluctant recurrence as:
there exists a folding point $x=(\ldots, x_{-1}, x_0)\in X$, an interval $J\subset I$ such that $x_0\in J^\circ$, 
and a monotone pull-back of $J$ along $x$.

\begin{lemma}\label{lem:infinite}
	Let $y\in \omega(c)$, $y\in U^\circ$ where interval $U\subset I$ and assume that for every $i\in\N$ the set $U$ can be monotonically 
	pulled-back along $(c_1, \ldots c_{n_i+1})$, where $U\ni c_{n_i+1}\neq y$. Then $U$ can be monotonically  
	pulled-back along some infinite backward orbit $(\ldots, y_{-2}, y_{-1}, y)$, where $y_{-i}\in\omega(c)$ 
	for every $i\in\N$.
\end{lemma}

\begin{proof}
	Note that the preimage of every interval consists of at most two intervals. So for every $k\in\N$ it is possible 
	to find a maximal $U^k$ such that $T^k(U^k)=U$ and $U^k$ contains $c_{n_i-k+1}$ for infinitely many 
	$i\in\N$ ($k<n_i+1$). Since we assumed that $U$ can be monotonically pulled-back along $(c_1, \ldots, c_{n_i+1})$ 
	for every $i\in\N$, $U^k$ can be chosen such that $c\not\in U^k$ for every $k\in\N$. 
	Thus $U, U^1, U^2, \ldots$ is a monotone pull-back of $U$ along an infinite backward 
	orbit $(\ldots, y_{-2}, y_{-1}, y)$, where $T^k(y_{-k})=y$, $y,y_{-k}\in U^k$ and $y_{-k}\in\omega(c)$ for every $k\in\N$.
\end{proof}

Next we characterise when $\F=\E$. Some partial results are already known. Namely,
$\F=\E$ when $Q(k)\to\infty$ and if $T|_{\omega(c)}$ is one-to-one, see \cite{Al}. 
However, there are examples which show that the converse does not hold; $\F=\E$ does not imply $Q(k)\to\infty$ or $T|_{\omega(c)}$ being one-to-one, see \cite{AlBr}. 
The question of distinguishing endpoints within the set of folding points originated from the study 
of infinitely renormalisable unimodal maps $f$. Then $f|_{\omega(c)}$ is conjugate to an adding machine (see \cite{dMvS}) 
and $\F=\E$. However, having an embedded adding machine (which can also happen in 
non-renormalisable case, see \cite{BKM} for the construction of strange adding machines) does not suffice to have $\F=\E$. 

\begin{theorem}\label{thm:FisE}
For $X'$ it holds that $\F=\E$ if and only if $c$ is persistently recurrent.
\end{theorem}

\begin{proof}
If $c$ is reluctantly recurrent, there exists a folding point $x=(\ldots, x_{-1}, x_0)\in X$, an interval $J$ such that 
$x_0\in J^\circ$, and an infinite monotone pull-back $(J_n)_{n\in\N_0}$ of $J$ along $x$. 
Note that $\underleftarrow{\lim}\{J_n, f|_{J_n}\}$ is an arc in $X$ and it contains $x$ in its interior, 
thus $x$ is not an endpoint.

For the other direction, let $c$ be persistently recurrent and assume that there is a folding point $x=(\ldots, x_{-1}, x_0)\in X'$ which is not an endpoint. Without loss of 
generality we can assume that $x$ is contained in the interior of its basic arc. Otherwise, we use 
$\sigma^{-j}(x)$ for some $j\in\N$ large enough. Let $A$ be a subset of the basic arc of $x$ such that $\partial \pi_0(A)\cap \orb(c)=\emptyset$ 
and such that $x\in A^\circ$. Let $A_k:=\pi_k(A) \subseteq[c_2, c_1]$ 
for every $k\in\N_{0}$. Denote by $J=A_0$ and by $(J_n)_{n\in\N_0}$ the pull-back of $J$ along $x$. 
Note that $A_n\subset J_n$ for every $n\in\N_0$. Since $c$ is persistently recurrent, 
there exists the smallest $i\in\N$ such that $c\in J^\circ_{i}$. Thus $A_0=J_0,A_1=J_1,\ldots A_{i-1}=J_{i-1}$ but $A_i\subsetneq J_i$. 
Since $c\not\in A_n^\circ$ for 
every $n\in\N$ (because otherwise $\partial A_0\cap \orb(c)\neq\emptyset$), it follows that $c$ is an endpoint of $A_{i}$, since $T(c)=c_1\in \partial[c_2,c_1]$ (note that it is important here that $A_k \subseteq[c_2, c_1]$ for all $k\in\N_0$). But then $c_i$ is an endpoint of $A_0=A$, 
which is a contradiction.
\end{proof}

\begin{remark}
This actually proves that if $c$ is persistently recurrent, then no non-degenerate basic arc contains a folding point in its interior. 
So the possible folding points in such $X'$ are either degenerate basic 
arcs or flat endpoints. In the rest of this 
section we show that both types can occur and show how that relates to the condition $Q(k)\to\infty$.
\end{remark}

\begin{remark}
Note that $Q(k) \to \infty$ implies that $c$ is persistently recurrent
(but not vice versa, see \cite[Proposition 3.1]{Br5}). However,
$Q(k) \to \infty$ is equivalent to $|D_n| \to 0$.
\end{remark}

\begin{proposition}\label{prop:degarc}
If $Q(k)\to\infty$, then all folding points are degenerate basic arcs (so either spiral or nasty endpoints).
\end{proposition}

\begin{proof}
Since $Q(k)\to\infty$, also $|D_n|\to 0$ as $n\to\infty$ and $c$ is persistently recurrent, 
so every folding point is an endpoint. If $x$ is an endpoint, 
then $\tau_L(x)=\infty$ or $\tau_R(x)=\infty$. 
Assume without loss of generality that $\tau_L(x)=\infty$, so $N_L(\ovl x)$ is an infinite set. 
Since $A(\ovl x)\subseteq\cap_{l\in N_L(\ovl x)}D_l$ and $|D_n|\to 0$, it follows that $A(\ovl x)$ is degenerate.
\end{proof}

\begin{remark}\label{rem:degendpoin}
It follows immediately from Proposition~\ref{prop:endpt} that every degenerate basic arc is an endpoint of $X$.
\end{remark}

\noindent
\begin{problem} Is it true that if $Q(k)\to\infty$ and $T$ is not infinitely renormalisable, 
then all the folding points are spiral points?
\end{problem}

\begin{remark} Let us comment the preceding problem. 
Nasty points are realized as nested intersections of non-arc subcontinua, see Proposition~\ref{prop:charaNasty}. So if the subcontinua 
of $X$ are simple enough, nasty points cannot exist. In \cite{BrBr, Br3} the authors give conditions which imply that
all subcontinua are arc+ray continua. In \cite[Theorem 1.1]{Br3} it is shown that if additionally
$Q(k+1) > Q(Q(k)+1)+1$ for all sufficiently large $k$, then all proper subcontinua are points, 
arcs and $\sin (1/x)$-continua. So if this technical assumption can 
be removed, the answer to the problem above is yes.
\end{remark}

\begin{proposition}\label{prop:Qntoinfty}
If $Q(k)\not\to\infty$, then there exists a folding point which is contained in a non-degenerate basic arc.
\end{proposition}
\begin{proof}
	If $Q(k)\not\to\infty$, then $|D_n|\not\to 0$ so there exists a sequence $(n_i)_{i\in\N}$ and $\delta>0$ such that $|D_{n_i}|>\delta$ for every $i\in\N$. For every $n\in\N$ there exists a basic arc $A_n\subset X'$ with $\pi_0(A_n)=D_n$, 
	\eg take $A_n=A(\ovl x)$ for $\ovl x=\ldots 111\nu_1\ldots\nu_{n-1}$. 
	The sequence of basic arcs $\{A_n\}_{n\in\N_0}\subset X'$ which project to $D_{n_i}$ accumulate on some basic arc $B\subset X'$ with $|\pi_0(B)|\geq\delta$. Note that such a basic arc $B$ must contain a folding point (which can be an endpoint of $B$ or in the interior of $B$).
\end{proof}

Since $\F=\E$ if $c$ is persistently recurrent, we obtain the following statement if we apply $\sigma^i$ for $i\in \Z$ to a flat endpoint provided by Proposition~\ref{prop:Qntoinfty}.

\begin{corollary}
If $Q(k)\not\to\infty$ and $c$ is persistently recurrent, then there exist infinitely many flat endpoints in $X$.
\end{corollary}

\begin{proposition}
	If $Q(k)$ is unbounded and $T|_{\omega(c)}$ is one-to-one, then there exist infinitely many folding points which are degenerate basic arcs (so either spiral or nasty).
\end{proposition}
\begin{proof}
	Given $n\in \N$, recall that if $S_k < n\leq S_{k+1}$, then $\beta(n) = n-S_k$; similarly 
	define $\gamma(n) := S_{k+1}-n$. As $Q(k)$ is unbounded, we may take an increasing sequence 
	$\{n_j\}_{j\geq 1}$ such that $\beta(n_j) = n_{j-1}$ and $\gamma(n_j) > \beta(n_j)$ for 
	all $j\in \N$. Since $D_n\subset D_{\beta(n)}$ for every $n$, then $D_{n_j} \subset D_{n_{j-1}}$ for all $j\in \N$. Also, since $T$ is locally eventually onto and $\beta(n_j)\to\infty$ implies $\gamma(n_j)\to \infty$, it follows that $|D_{n_j}|\to 0$. 
	Thus $\cap_{j>1} D_{n_j} = \{x_0\}\subset \omega(c)$. Note that because $x_0$ has a unique preimage in $\omega(c)$ and because $|D_{n_j-1}|\to 0$ as $j\to \infty$, it follows that the unique preimage $x_{-1}$ of $x_0$ in $\omega(c)$ must lie in $D_{n_j-1}$ for all large $j$. Similarly, there is a unique $i$-th preimage $x_{-i}$ of $x_0$ in $\omega(c)$ that must lie in $D_{n_j-i}$ for all large $j$ and for all $i = 1, 2, \ldots, \beta(n_j)-1$. Then $x = (\ldots, x_{-i}, \ldots, x_{-2},x_{-1},x_0)\in \mathcal{F}$ with either $\tau_{L}(\ovl x)=\infty$ or $\tau_{R}(\ovl x)=\infty$. Without loss of generality, there exists a subsequence $\{n_{j_k}\}$ such that $n_{j_k} \in \NL(\ovl x)$ for all $k\in \N$. 
	Since $A(\ovl x)\subseteq\cap_{l\in N_L(\ovl x)}D_l \subseteq \cap_{k\in \N}D_{n_{j_k}}$ 
	and $|D_{n_{j_k}}|\to 0$, it follows that $A(\ovl x)$ is degenerate. Thus we found a folding point being a degenerate basic arc. We apply $\sigma^i$ for $i\in \Z$ to get countably infinitely many such endpoints.
\end{proof}

It thus follows that there exist examples of tent maps with $\omega(c) \neq [c_2,c_1]$ that contain flat endpoints and spiral and/or nasty endpoints; see for an example \cite[Example 3.10]{AlBr}. 
Note that in that example $Q(k)\not\to \infty$ but $T|_{\omega(c)}$ is one-to-one and still $\mathcal{F}=\mathcal{E}$.

\begin{proposition}\label{prop:spiral2}
Assume that $Q(k)\to\infty$ and $Q(k) \leq k-2$ for all $k$ sufficiently large. 
Then $\E_S$ is infinite. 
\end{proposition}

\begin{proof}
Define recursively a sequence $(k_i)_{i \geq 1}\subset \N_0$ by setting $k_i = \min\{ k : Q(k) > k_{i-1}-1\}$.
Then obviously $Q(k_i)>k_{i-1}-1$ and $Q(k_i-1) \leq k_{i-1}-1$. So by assumption,
$Q(Q(k_i-1)+1) < Q(k_i-1) \leq k_{i-1}-1$. Therefore $(k_i)_{i \geq 1}$ satisfies \eqref{eq:Qcond2}
from Remark~\ref{rem:Qcond2} which gives the existence
of a subcontinuum $H$ that is a direct spiral or a basic $\sin(1/x)$-continuum.
However, since $Q(k) \to \infty$ and hence $|D_n| \to 0$, the latter is not possible.
Therefore $\E_S \neq \emptyset$, and since $\sigma^j(H) \neq H$ for all $j \in \Z \setminus \{ 0 \}$,
$\E_S$ is infinite.
%
\end{proof}

\begin{remark}[Example 3.5 from \cite{Al}]
	Consider the symmetric tent map $T$ with kneading map $$Q(k)=\begin{cases}
	0 & \text{ if } k\in \{1,2,4\}, \\ 1 & \text{ if } k = 3, \\ 3\ell - 4 & \text{ if } k = 3\ell-1 \text{ or } 3\ell+ 1 \text{ and } \ell \geq 2, \\ 3\ell - 2 & \text{ if } k = 3\ell \text{ and } \ell \geq 2.
	\end{cases}$$
	Take $(k_i)_{i\geq 3} = (3i-1)_{i\geq 3}$.  Then \eqref{eq:Qcond} holds, and as $Q(k)\to\infty$, it follows that $\mathcal{E}_S\neq \emptyset$. We note that $T$ is non-renormalisable and $T|_{\omega(c)}$ is topologically conjugate to the triadic adding machine. This is in contrast to the infinitely renormalisable maps which have $\mathcal{E}_S = \emptyset$  (cf. Theorem~\ref{thm:nastyIR}).
\end{remark}

\subsubsection{Nasty endpoints}\label{sec:nasty} In this subsection we prove the existence of nasty points in tent inverse limit spaces. Furthermore, we also prove that nasty points are the only endpoints that appear in the core inverse limit spaces of infinitely renormalisable logistic maps. At the end of the subsection we provide some general results about existence of specific endpoints in tent inverse limits.

\begin{definition}
Given a continuum $K$, we call a point $x \in K$ a \emph{nasty point}
if its arc-component is degenerate.
The \emph{set of all nasty points} in the inverse limit space $X$ is denoted by $\E_N$.
\end{definition}

Note that every nasty point in unimodal inverse limit space $X$ is automatically an endpoint since it lies in a degenerate basic arc, see Remark~\ref{rem:degendpoint}. We continue with some more general facts about nasty points in (chainable) continua.

\begin{lemma}\label{lem:nested}
Let $K$ be a non-degenerate continuum.
For every $x\in K$ there exists a nested sequence of non-degenerate subcontinua 
$\{H_i\}_{i\in \N}\subset K$ such that  $\cap_{i\in\N} H_i=\{x\}$.
\end{lemma}

\begin{proof}
If $K$ is decomposable, then clearly there is a proper subcontinuum $K_1 \owns x$.
If $K$ is indecomposable, the composant of $x$ is dense in $K$ and thus there exists a proper 
subcontinuum $K_1\subset K$ such that $x\in K_1$. 
Let the set $\{H_{\lambda}\}_{\lambda\in \Lambda}$ consist of all proper subcontinua of $K$ containing $x$. The set $H:=\cap_{\lambda\in \Lambda} H_{\Lambda}$ is a continuum. If $H=\{x\}$, we are done since the intersection can be taken nested.\\
	Assume by contradiction that $H$ is a non-degenerate continuum. Then $H$ is indecomposable, because otherwise we could find a non-degenerate continuum $H'\subset H$ such that $H'\neq H_{\lambda}$ for all $\lambda\in \Lambda$. But if $H$ is indecomposable, the composant of $x$ is dense in $H$ so there is a subcontinuum $x\in H''\subsetneq H$, a contradiction.  
\end{proof}

We have the following characterisation of nasty endpoints in an arbitrary chainable indecomposable continuum $K$.

\begin{proposition}\label{prop:charaNasty}
Let $x\in K$ be an endpoint of a non-degenerate chainable continuum $K$. Then $x$ is not contained in an arc of $K$ if and only if there exists a nested sequence of non-degenerate 
subcontinua $\{H_i\}_{i\in \N}\subset K$ such that $\cap_{i\in\N} H_i=\{x\}$ and $H_i$ is not 
arc-connected for all sufficiently large $i\in\N$. 
\end{proposition}

\begin{proof}
Assume that a point $x$ is not contained in an arc.
By Lemma~\ref{lem:nested} there exists a nested sequence of non-degenerate subcontinua $H_i\subset K$ 
such that $\{x\}=\cap_{i\in\N} H_i$. If $H_i$ is arc-connected for some $i\in\N$, then there 
exists an arc $x\in A\subset K$, a contradiction.\\
Conversely, assume by contradiction that an endpoint $x$ is contained in a non-degenerate arc $A$ and take a nested sequence of non-degenerate non-arc subcontinua $H_i\subset K$ such that $\{x\}=\cap_{i\in\N} H_i$. Since $x$ is an endpoint, it holds that $H_i\subset A$ for large enough $i$, which gives a contradiction.
\end{proof}

\begin{remark}
	Note that Proposition~\ref{prop:charaNasty} fails to be true if $x\in K$ is not an endpoint. Say that $P'=P\cup A$ where $P$ is the pseudo-arc and $A$ an arc and $P\cap A=\{x\}$. Then $x$ is not an endpoint of $P'$, however $\{x\}=\cap_{i\in\N} H_i$ where $H_i\subset P$ are the pseudo-arcs. 
	\\ Furthermore, the assumption of chainability in Proposition~\ref{prop:charaNasty} is needed in the definition of an endpoint.
	 Suppose that we use Lelek's definition of an endpoint of a continuum $K$ (a point in $K$ is an endpoint, if it is an endpoint of every arc contained in $K$). However, Proposition~\ref{prop:charaNasty} with this definition of an endpoint fails to be true since \eg $x\in A\subset P'$ from the last example is an endpoint.
\end{remark}

For infinitely renormalisable quadratic maps we have 
the following simple characterisation of folding points: they are all nasty endpoints.

\begin{theorem}\label{thm:nastyIR}
 If $T = Q_a$ is infinitely renormalisable, then
 $X'$ contains a Cantor set of nasty endpoints. There are no other folding points, i.e., $\F = \E_N$.
\end{theorem}

\begin{proof}
 Since $T$ is infinitely renormalisable, there is a nested sequence $J_i$ of $p_i$-periodic cycles of intervals
 $J_i = \{ J_{i,k} \}_{k=0}^{p_i-1}$, where $J_{i,0} \owns c$, $T(J_{i,k}) = J_{i,k+1}$ for $0 \leq k < p_i$
 and $T(J_{i,p_i-1}) = J_{i,0}$. We have $\omega(c) = \cap_i J_i$ and it is a Cantor set.
 Associated to $J_{i,k}$ are subcontinua
 $$
 G_{i,k} = \{ x \in X : \pi_{j p_i}(x) \in J_{i,k} \text{ for all } j \geq 0\},
 $$
 and each $G_{i,k}$ is homeomorphic to the inverse limit space of the $i$-th renormalisation of $T$,
 and hence non-degenerate and not arc-connected (since they are not arcs).
 Furthermore, $\diam(G_{i,k}) \to 0$ as $i \to \infty$.
 Therefore we have an uncountable collection of sequences $(k_i)_i$ with $G_{i,k_i} \supset G_{i+1, k_{i+1}}$
 such that $\cap_i G_{i,k_i}$ is a single point satisfying the characterisation of a nasty endpoint.
 
 If $x \in X'$ is not of this form, then there are $j,i$ such that the projection $\pi_j(x) \notin J_i$.
 But that means that $\pi_j(x) \notin \omega(c)$, so $x$ is not a folding point.
 Since the set of folding points is a Cantor set  (i.e., $\underleftarrow{\mathrm{lim}}(\omega(c),T|_{\omega(c)})$ is nowhere dense and perfect as in the argument at the start of Section~\ref{sec:endpoint}), 
 it follows that $\F = \E_N$ is the Cantor set.
\end{proof}

Now we return to non-renormalisable maps.
The following result of Barge, Brucks and Diamond \cite{BBD} gives a way to find nasty points.

\begin{proposition}[Theorem 4 in \cite{BBD}]\label{prop:BBD}
	For a dense $G_{\delta}$ set of parameters $s\in[\sqrt 2, 2]$ it holds that
	every open set in $X'$ contains a homeomorphic copy of every tent inverse limit space. 
\end{proposition}

We denote this $G_\delta$ set of parameters by $\mathcal A$ 
(it is originally  denoted by $A$ in \cite{BBD}). 
The characterising property for $s\in \cA$ is that for any $a\in[c_2, c]$ and 
$\delta>0$ there exist $n\in\N$ and $c_2<a_s<b_s<c_1$ such that $T_s^n(c_2)\in (a-\delta, a+\delta)$, 
$T_s^n(a_s)=c_2$, $T_s^n(b_s)=c$ and $T_s^n$ is monotone on $[c_2, a_s]$ and $[a_s, b_s]$. 
Note that for every parameter $s\in \mathcal A$ the critical orbit is dense in the core. The properties of $\mathcal A$ are further discussed in Appendix~\ref{Appendix}.
	
  The following statement interprets Proposition~\ref{prop:BBD} in a different setting.

\begin{corollary}\label{cor:nasty_pts_exist}
For $s\in \cA$ there exists a dense set of nasty endpoints $x\in X'$. Furthermore, the cardinality $\#(\E_N) = 2^{\aleph_0}$.
\end{corollary}
\begin{proof}
	Proposition~\ref{prop:BBD} gives a dense set of points $x$ for which 
	there exist non-arc subcontinua $H_i\subset X'$, $H_{i+1}\subset H_i$ for every $i\in\N$, such that $\diam(H_i)\to\infty$ 
	as $i\to\infty$ and $\cap_{i\in\N}H_i=\{x\}$. Since every such $x$ is a degenerate basic arc it is 
	automatically an endpoint of $X'$ by Proposition~\ref{prop:endpt}. 
	The characterisation of nasty points in Proposition~\ref{prop:charaNasty} implies that every such $x$ is a nasty endpoint of $X'$. Note that the construction allows uncountably many nested sequences producing nasty endpoints.
\end{proof}

Next we give an analogue of Proposition~\ref{prop:uncountablyEndpts} 
for the sets of endpoints $\E_F$, $\E_S$, and $\E_N$.

\begin{proposition}\label{prop:FSuncountable}	
	If $s\in \cA$, then the sets $\E_F$, $\E_N$ and $\E_S$ are uncountable when non-empty.
\end{proposition}

\begin{proof}
	If $\omega(c)$ contains an interval, then $\omega(c) = [c_2,c_1]$ or $T$ is renormalisable, and the 
	deepest renormalisation is a unimodal map with $\omega(c) = [c_2,c_1]$.
	Hence we can assume that $\omega(c) = [c_2,c_1]$. 
	
	First we claim that $\{c_{S_k}: k\in\N, Q(k)\leq 1\}$ is dense in $[c_2, c_1]$ if $s\in \cA$. 
	Note that since $\{c_{S_k}:k\in\N\}$ is dense in $[\hat z_1,c_1]$ (cf. Proposition~\ref{lem:strongerBBD}) it follows that 
	$\{c_{S_k}: k\in\N, Q(k)=0\}$ is dense in $[c_2, c]$, and since $\{c_{S_k}:k\in\N\}$ 
	is dense in $[\hat z_2, \hat z_1]$, it follows that $\{c_{S_k}: k\in\N, Q(k)=1\}$ is dense in $[c,c_1]$.
	
	So we can find $k_1$ so that $z_1 \in (c_{S_{k_1}}, c)$ (so $Q(k_1+1) \leq 1$) and $Q(k_1-1)\leq 1$.
	Assume now by induction that $k_{i-1}$ is chosen such that
	$z_1 \in (c_{S_{k_{i-1}}}, c)$ (so  $Q(k_{i-1}+1) \leq 1$) and $Q(k_{i-1}-1)\leq 1$.
	Next choose $k_i > k_{i-1}$ such that  $Q(k_i-1)\leq 1$, $c_{S_{k_{i}-1}} \in \Upsilon_{k_{i-1}}$ 
	(so $Q(k_i) = k_{i-1}$), and in fact so close to $z_{k_{i-1}-1}$ that
	$z_1 \in (c_{S_{k_i}}, c)$. Note that this is possible since  $ f^{S_{Q(k_i)}}((z_{k_{i-1}-1} , z_{k_{i-1}}))=(c_{S_{Q(k_{i-1})}}, c)=(c_{S_{k_i-2}},c)\ni z_1$.
	Since we have a choice at each induction step, we obtain this way uncountably many sequences 
	$(k_i)_{i \geq 1}$ with $Q(k_i) = k_{i-1}$ for $i \geq 2$, $Q(Q(k_i-1)+1)$ bounded, and $\liminf_i Q(k_i+1) \leq 1$.
	
	Alternatively, we can choose by induction  $k_i > k_{i-1}$ such that $Q(k_i-1)\leq 1$, $c_{S_{k_{i-1}-1}} \in \Upsilon_{k_{i-1}}$ 
	(so $Q(k_i) = k_{i-1}$), and in fact so close to $z_{k_{i-1}}$ that
	$c_{S_{k_i}} \in (z_i, c)$, so $Q({k_i+1})> i$.
	Since we have a choice at each induction step, we obtain this way uncountably many sequences 
	$(k_i)_{i \geq 1}$ with $Q(k_i) = k_{i-1}$ for $i \geq 2$, $Q(Q(k_i-1)+1)$ bounded, and $\lim_i Q(k_i+1) = \infty$.
	
	Thus by Proposition~\ref{prop:spiral} there are uncountably many spiral points and uncountably many
	flat endpoints (at the bars of basic $\sin(1/x)$-continua). 
	Finally, Corollary~\ref{cor:nasty_pts_exist} gives that the set of nasty endpoints is uncountable as well.
\end{proof}

\begin{problem}
	If $\omega(c) =[c_2,c_1]$ are 
	the sets $\E_N, \E_F$ and $\E_S$ always uncountable when non-empty?
\end{problem}

While $s\in \cA$ guarantees that $X$ contains a copy of every continuum that arises as an inverse limit space of a core tent map, there is no known complete generalisation of maps with this property. However, we are able to show that this property cannot hold if $\omega(c)\neq [c_2,c_1]$.

\begin{proposition}\label{prop:Raines}
If $\omega(c)\neq[c_2,c_1]$ then $X$ does not contain a copy of every inverse limit space from the parametrised tent family.
\end{proposition}

\begin{proof}
	We only need to prove that in the case when $T$ is such that $\omega(c)$ is the Cantor set and $c$ is recurrent we cannot find every inverse limit space of the core tent map family in $X$. 
	Let $X$ be a tent inverse limit space so that $\omega(c)$ is a Cantor set and $c$ is recurrent and assume that there exists $H\subset X$ so that $H$ is homeomorphic to a tent inverse limit space $Y$ with critical orbit $\tilde{c}$ dense in $[\tilde{c}_2,\tilde{c}_1]$. Since it follows from Proposition~\ref{prop:foldpts} that every point from $H$ is a folding point, there exists a non-degenerate arc $A\subset \tilde{\cR}\subset H$ such that every $x\in A$ is a folding point. Therefore, there exists an interval $\pi_0(A)\subset [c_2,c_1]$ with $|\pi_0(A)|>0$ and such that $\pi_0(A)\subset \omega(c)$. 
	Since the Cantor set is nowhere dense, we have a contradiction.
\end{proof}

Despite the fact that we have proven the existence of nasty points in unimodal inverse limit spaces, our knowledge about 
them is limited. 
Because $s\in \cA$ if and only if the set $\{c_{S_k}: k\in \N\}$ is dense in $[c_2,c_1]$ (cf. Proposition~\ref{lem:strongerBBD}), it is not even known if nasty points always exist when $\omega(c)=[c_2,c_1]$. 
If $c$ is recurrent and $\omega(c)$ is the Cantor set, then there is no known characterisation of subcontinua of $X$. 
It is a priori possible that there exist $X$ that contain complicated subcontinua which are realized as nested intersections of other non-arc unimodal inverse 
limit spaces with recurrent critical orbit for which $\omega(c)$ is a Cantor set.
There are only some partial results on conditions precluding nasty endpoints. For example, constructions in \cite{BrBr} and \cite{Br3} provide examples of inverse limit spaces of tent maps that have exactly points, arcs, rays, arc+rays continua and/or continua homeomorphic to core tent inverse limit spaces with finite critical orbits; in these cases, there are no nasty points. Thus we pose the following problem.


\begin{problem}
	Give necessary conditions on the critical point $c$ so that the corresponding inverse limit space $X$ contains nasty points.
\end{problem}

%

To make this problem easier to study, one approach is to first answer the following problem.

\begin{problem}
	Give a symbolic characterisation of nasty points in $X$.
\end{problem}

\appendix
\section{Characterising $s\in\cA$}\label{Appendix}

We want to characterise $s\in\cA$ in terms of a kneading map/sequence. It turns out that $s\in\cA$ if and only if $\{c_{S_k}: k\in\N\}$ is dense in $[c_2, c_1]$, see Proposition~\ref{lem:strongerBBD}. Naturally, if $\{c_{S_k}: k\in\N\}$ is dense in $[c_2, c_1]$, so is $\orb(c)$.
The following proposition shows that the converse does not hold, thus giving a positive answer to Question 6.4.8.\ in \cite{BB}. Specifically, we cannot claim that $s\in\cA$ if and only if $\orb(c)$ is dense in $[c_2, c_1]$. That does not mean that the self-similarity result of \cite{BBD} does not hold for slopes for which $\orb(c)$ is dense, with a possibly more complicated construction. We pose the following problem:

\begin{problem}
	If $\orb(c)$ is dense in the core, does every neighbourhood of every point in $X$ contains a copy of every other tent inverse limit?
\end{problem}

\begin{proposition}\label{prop:counterexample}
	There exists a tent map with a dense critical orbit, such that $\{c_{S_k}\}_{k\in\N_0}$
	is not dense in $[c_2,c_1]$.
\end{proposition}

\begin{proof}
	According to Hofbauer \cite{Hof,Br0}, a kneading sequence is admissible if and only if its kneading map 
	$Q:\N \to \N_0$ exists and satisfies 
	\begin{equation} \label{eq:hofb}
	Q(k) < k \text{ and } \{ Q(Q^2(k)+j) \}_{j \geq 1} \preceq_{lex} \{ Q(k+j) \}_{j\geq 1} \qquad \text{ for all } k \geq 1.
	\end{equation}
	Here $\preceq_{lex}$ is the lexicographical order on sequences of natural numbers and $Q(0) = 0$ by convention.
	Taking $k-1$ instead of $k$ in the left hand side of \eqref{eq:hofb},
	we have $Q(Q^2(k-1)+1) \leq Q^2(k-1) \leq Q(k-1)-1$.
	Therefore, regardless of what $Q(j)$ is for $j < k$, one can always set $Q(k)  = m$
	for any $m > Q(k-1)-1$. We can also set $m=Q(k-1)-1$ provided we take $Q(k+1)$ sufficiently large, \eg $Q(k+1)\geq Q(Q^2(k)+1)$, where if $Q(k+1)=Q(Q^2(k)+1)$ we have to take $Q(k+2)\geq Q(Q^2(k+1)+1)$, etc.
	
	The map is renormalisable if and only if there is some $k \geq 2$ such that
	$Q(k+j) \geq k-1$ for all $j \geq 0$, see \cite[Proposition 1(iii)]{Br0}, so assuming that $Q(k) \leq k-2$ for all $k \geq 2$
	prevents renormalisation.
	
	Given a word $w \in \{ 0, 1\}^n$, let $w'$ be the same word with the last letter swapped.
	Suppose that the kneading sequence $\nu$ is known up to the cutting time $S_k$.
	Let $\W_k$ denote the collection of the words $w$ such that both $w$ and $w'$ appear in
	$\nu_1 \dots \nu_{S_k}$, with the last letters of $w$ and $w'$ both at cutting times.
	(Obviously, $w \in \W_k$ if and only if $w' \in \W_k$.)
	
	We extend $\nu_1 \dots \nu_{S_k}$ in steps, every time adding a new pair of admissible 
	words $w$ and $w'$ of shortest lengths so that their last letters appear at cutting times. 
	In addition, we make sure that $Q(l) \leq l-2$ (so as to avoid renormalisations)
	and also avoid using $Q(l)=1$.
	Since every admissible word is a prefix of a word in $\cup_k \W_k$,
	the limit sequence $\nu$ corresponds to a tent map with a dense critical orbit.
	However, since $Q(k) \neq 1$ for all sufficiently large $k$, $\{c_{S_k}\}_{k\in\N}$
	is not dense in the core.
	
	So let us give the details of the construction.
	Start with 
	$$
	\nu = \nu_1\dots \nu_7 = 1.0.0.0.101.  \qquad (\text{dots indicate cutting times, and } 7 = S_4).
	$$
	Thus $\W_4 = \{ 0,1,00,01, 100, 101\}$, so the shortest missing pair is $10,11$.
	In fact, $10$ already appears, but to accommodate $11$, we extend $\nu$ to
	$$
	\nu = 1.0.0.0.101.0.101.10001011.
	$$
	The extra block $101$ is there to assure that $Q(j) \leq j-2$.
	
	Now for the general induction step, let $v$ be (one of the) shortest admissible word(s)
	not yet appearing in $\W_k$ and such that $v'$ is admissible too.
	Let $w$ be the longest common prefix of $v$ and $v'$ such that $w \in \W_k$,
	so $v = wu$ and $v' = wu'$. By switching the role of $v$ and $v'$ if necessary, 
	we can assume that $u'$ has an even number of ones in it.
	Also let $1<n' < k$ be the smallest integer such that $w'$ appears as the suffix of  $\nu_1\cdots \nu_{S_{n'}}$. 
	(If $n'=k$, then extend $\nu$ by one block $\nu_1 \dots \nu_{S_{k-1}-1}\nu'_{S_{k-1}}$.)
	
	Now extend $\nu$ as
	\begin{eqnarray*}
		\nu &=& \underbrace{\nu_1 \dots \nu_{S_k}}_{\text{previous }\nu} . 
		\underbrace{\nu_1 \dots \nu_{S_{Q(k)-1}-1}  \nu'_{S_{Q(k)-1}} .
			\dots \nu_{S_{Q(k)-2}-1}  \nu'_{S_{Q(k)-2}} \dots \nu_1 \dots \nu_{S_2-1} \nu'_{S_2}}_{\text{block I}} . \\
		&&  \underbrace{\nu_1 \dots \nu_{S_{n'}-1}\nu'_{S_{n'}} .
			u'}_{\text{block II}} .\underbrace{\nu_1 \dots \nu_{S_r-1} \nu'_{S_r}}_{\text{block III}} . 
		\underbrace{\nu_1 \dots \nu_{S_k} \dots \dots \nu'_{S_2} \nu_1 \dots \nu_{S_{n'}-1}\nu'_{S_n'} u}_{\text{block IV}}.
	\end{eqnarray*}
	\begin{itemize}
		\item[I] By setting $Q(j)= Q(j-1)-1$ for successive $j \geq k+1$, we bring down $Q$ stepwise to $2$. This is admissible since $Q(Q^2(k+j-1)+1)\leq Q^2(k+j-1)\leq Q(k+j-1)-1=Q(k+j)$ for all $j$ as above.
		Also, according to \eqref{eq:hofb}, any value of $Q$ greater than 1 is allowed directly afterwards.
		\item[II] 
		We claim that the word $\nu_1\cdots \nu'_{S_{n'}}u'$ is admissible and that the last letter of the appearance of
		$u'$ is a cutting time (see below).
		Since $w$ is the suffix of $\nu_1\dots\nu_{S_{n'}-1}\nu'_{S_n'}$, we now have $v'$ appearing  with the last letter at a cutting time. 
		
In order to explain why $\nu_1\cdots \nu'_{S_{n'}}u'$ is admissible and that we indeed have cutting times in the word $\nu_1 \dots \nu_{S_{n'}-1}\nu'_{S_{n'}} .
u'.$ as denoted, we have to introduce some additional notation as follows.

Let $\rho(j) := \min\{k > j : \nu_k \neq \nu_{k-j}\}$ ($\rho(j)$ as defined here is unrelated with a fixed point $\rho$ used in the rest of the paper) and 
recall from \eg \cite{Br0} that the co-cutting times are the $\rho$-orbit
starting at $\min\{ j > 1 : \nu_j = 1\}$, whereas the cutting times are the $\rho$-orbit
starting at $1$. An admissibility condition equivalent to (\ref{eq:hofb}) is that the sequences of
cutting times and co-cutting times are disjoint (admissibility condition A3 in \cite{Br0}, see also \cite{Thun}).

Since $n'$ is chosen minimal, the largest co-cutting time before $S_{n'}$ is 
greater than $S_{n'}-|w|$ and in particular, $S_{Q^2(n')} < |w|$
(because when there is a co-cutting time between $S_{n'}$ and $S_{n'-1}$,
then $S_{n'} - S_{Q^2(n')}$ has to be a co-cutting time, see (3.12) in \cite{GreenBook}). 
The block $\nu_1\cdots \nu'_{S_{n'}}$ is admissible, and thus $S_{n'} - S_{Q^2(n')}$
must be a co-cutting time.

Since $v'$ is an admissible word, if we mark the $\rho$-orbits inside $v'$ starting
at entries $|w|$ and $|w|-S_{Q^2(n')}$, we find them disjoint.
Therefore, if we mark the $\rho$-orbits inside $\nu_1\cdots \nu'_{S_{n'}}u'$ starting
at entries $1$ and $S_{n'}-S_{Q^2(n')}$, we find them disjoint as well.
Therefore $\nu_1\cdots \nu'_{S_{n'}}u'$ is admissible, and the same argument applies
to $\nu_1\cdots \nu'_{S_{n'}}u$.
In particular, $|\nu_1\cdots \nu'_{S_{n'}}u'|$ must indeed be a cutting time,
and since $u'$ has an even number of ones by choice, block II indeed ends at a cutting time.

		\item[III] This extra block $\nu_1 \dots \nu_{S_r-1} \nu'_{S_r}$ is there to prevent us from having 
		$Q(j)=j-1$. We choose $r$ minimal such that the extension with this block is admissible.
		\item[IV] Here we added (previous $\nu$)+ block I+ block II with the last symbol switched, so $Q(j) = j-2$ which is always allowed. 
		We now have $v$ appearing with the last letter at a cutting time. 
	\end{itemize}
	We will now verify that $|u|$ is a cutting time and $|u| \neq 2$, so that we can conclude by induction that $Q(j) \neq 1$ for this extended $\nu$.

	\begin{figure}[!ht]
		\centering
		\begin{tikzpicture}[scale=0.5]
		\draw (0,0)--(10,0); 
		\draw (1,-0.1)--(1,0.1); \node at (1,-0.5) {\small $c_a$};
		\draw (3,-0.1)--(3,0.1); \node at (3,-0.5) {\small $z_j$};
		\draw (5,-0.1)--(5,0.1); \node at (5,-0.5) {\small $c$};
		\draw (9,-0.1)--(9,0.1); \node at (9,-0.5) {\small $c_b$};
		\node at (11.3,1.5) {\small $T^m$}; \draw[->] (10.5,2.8)--(10.5,0.2); 
		\draw (0,3)--(10,3); 	  
		\draw (1,2.9)--(1,3.1); \node at (3,3.4) {\small $Z$};
		\draw (5,2.9)--(5,3.1); \node at (5,2.5) {\small $z$};
		\draw (9,2.9)--(9,3.1); \node at (7,3.4) {\small $Z'$};
		\node at (11.8,4.5) {\small $T^{n'-m}$};  \draw[->] (10.5,5.8)--(10.5,3.2); 	  
		\draw (0,6)--(10,6); 	  
		\draw (1,5.9)--(1,6.1); \node at (1,5.5) {\small $z_{n'-1}$};
		\draw (5,5.9)--(5,6.1); \node at (5,5.5) {\small $z_{n'}$};
		\draw (9,5.9)--(9,6.1); \node at (9,5.5) {\small $c$};
		\end{tikzpicture}
		\caption{Illustration of the sets $Z,Z'$ and points $c_a, c_b$}
		\label{fig:illus}
	\end{figure}
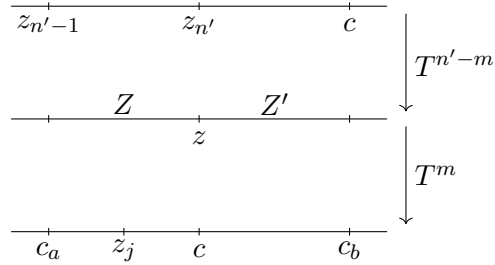

	First note that $w$ and $w'$ correspond to two adjacent cylinder sets $Z$ and $Z'$ of length $m := |w|$,
	and with some $z \in T^{-m}(c)$ as common boundary point.
	Thus there are integers  $1 \leq a,b < n'$ such that $T^m:Z \cup Z' \to [c_a, c_b]$
	is monotone onto and $[c_a,c_b] \owns c$, so $a$ and $b$ are in fact cutting times.
	Assume without loss of generality that $T^m(Z) = [c_a,c]$.
	Since both $wu$ and $wu'$ are admissible (and are the shortest words of the form $wU$ $wU'$ with this property),
	there is a closest precritical point $z_j \in [c_a,c]$ and $S_j = |u|$.
	Since $a = S_i < S_k$ is a cutting time with
	$Q(i) = j$, we get by induction $Q(i) \neq 1$, so $|u'| = S_j \neq 2$.
\end{proof}

The previous proposition in combination with the following show that one cannot use \cite{BBD} for showing that all dense critical orbit cases have the self-similarity property.

\begin{proposition}\label{lem:strongerBBD}
	The set $\{c_{S_k}: k\in\N\}$ is dense in $[c_2, c_1]$ if and only if $s\in\cA$. 
\end{proposition}

\begin{proof}
	Recall that $\{ S_k \}_{k \geq 0}$ denotes the set of cutting times of $T$, and
	$z_k, \hat z_k:= 1-z_k$ are the closest precritical points, i.e.,
	$T^{S_k}(z_k) = T^{S_k}(\hat z_k) = c$ and $T^j([z_k, \hat z_k]) \not\owns c$ for $0 < j < S_k$. Also recall that $s\in \cA$ if for any $a\in[c_2, c]$ and 
	$\delta>0$ there exist $n\in\N$ and $c_2<a_s<b_s<c_1$ such that $T_s^n(c_2)\in (a-\delta, a+\delta)$, 
	$T_s^n(a_s)=c_2$, $T_s^n(b_s)=c$ and $T_s^n$ is monotone on $[c_2, a_s]$ and $[a_s, b_s]$. 
	
	Assume that $\{c_{S_k}: k\in\N\}$ 
	is dense in $[c_2, c_1]$. Fix $a\in[c_2, c]$ and $\delta>0$. 
	Let $c_{-1} = \hat z_0$ and denote by $c_{-2}$ the point in $(c_{-1}, c_1)$ such that 
	$T^2(c_{-2})=c$, if such a point exists. Otherwise take $c_{-2}=c_1$. 
	Since $T^3([c_{-1}, c_{-2}])\supset[c_2, c]$, there exists $x\in [c_{-1}, c_{-2}]$ such that $T^3(x)=a$. Find $k\in\N$ such that $c_{S_k}\in (x-\delta/{s^3}, x+\delta/{s^3})$. 
	Then (see Figure~\ref{fig:bbd}) there are $y_1<y_2<c_{S_k}$ such that $T^3(y_1)=c, T^3(y_2)=c_2$, $T^3(c_{S_k})$ is in the $\delta$-neighbourhood of $a$, and $T^3$ is linear on $[y_1, y_2]$ and $[y_2, c_{S_k}]$. Also, since $S_k$ is a cutting time, there is an interval $[z, c]$ such that $T^{S_k}([z, c])=[y_1, c_{S_k}]$ is one-to-one, and thus the conditions in the definition of $\mathcal A$ are satisfied for $n=S_k+1$, $b_s=T^2(z)$, and $a_s\in[c_2, b_s]$ the unique point such that $T^{S_k+1}(a_s)=c_2$.
	
	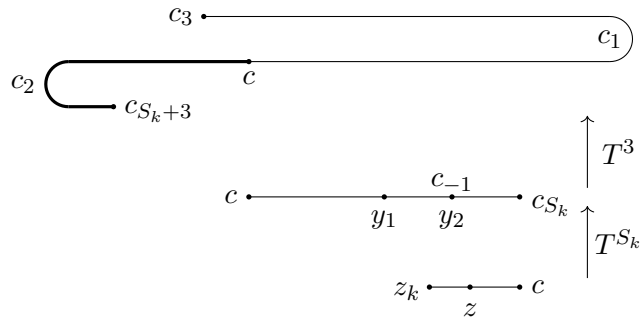
\begin{figure}[!ht]
		\centering
		\begin{tikzpicture}[scale=6]
		\draw (0.8,0)--(1,0);
		\draw[->] (1.15,0.02)--(1.15,0.18);
		\draw (0.4,0.2)--(1,0.2);
		\draw[->] (1.15,0.22)--(1.15,0.38);
		\draw[very thick] (0.1,0.4)--(0,0.4);
		\draw[very thick,domain=90:270] plot ({0+0.05*cos(\x)}, {0.45+0.05*sin(\x)});
		\draw[very thick] (0,0.5)--(0.4,0.5);
		\draw (0,0.5)--(1.2,0.5);
		\draw[domain=270:450] plot ({1.2+0.05*cos(\x)}, {0.55+0.05*sin(\x)});
		\draw (1.2,0.6)--(0.3,0.6);
		\node at (1.22,0.1) {\small $T^{S_k}$};
		\node at (1.22,0.3) {\small $T^{3}$};
		\node[circle,fill, inner sep=0.7] at (0.8,0){};
		\node[circle,fill, inner sep=0.7] at (0.89,0){};
		\node[circle,fill, inner sep=0.7] at (1,0){};
		\node[circle,fill, inner sep=0.7] at (0.4,0.2){};
		\node[circle,fill, inner sep=0.7] at (0.7,0.2){};
		\node[circle,fill, inner sep=0.7] at (0.85,0.2){};
		\node[circle,fill, inner sep=0.7] at (1,0.2){};
		\node[circle,fill, inner sep=0.7] at (0.1,0.4){};
		\node[circle,fill, inner sep=0.7] at (0.4,0.5){};
		\node[circle,fill, inner sep=0.7] at (0.3,0.6){};
		\node at (0.75,-0.01) {\small $z_k$};
		\node at (0.89,-.05) {\small $z$};
		\node at (1.04,0) {\small $c$};
		\node at (0.36,0.2) {\small $c$};
		\node at (0.7,0.15) {\small $y_1$};
		\node at (0.85,0.15) {\small $y_2$};
		\node at (1.07,0.18) {\small $c_{S_k}$};
		\node at (0.85,0.23) {\small $c_{-1}$};
		\node at (0.2,0.39) {\small $c_{S_k+3}$};
		\node at (-0.1,0.45) {\small $c_2$};
		\node at (0.4,0.465) {\small $c$};
		\node at (1.2,0.55) {\small $c_1$};
		\node at (0.25,0.6) {\small $c_3$};
		\end{tikzpicture}
		\caption{A step in the proof of Proposition~\ref{lem:strongerBBD}.} 
		\label{fig:bbd}
	\end{figure}
	For the other direction, take $s\in\cA$, and assume by  contradiction that $\{c_{S_k}: k\in\N\}$ is not dense in $[c_2, c_1]$. Note that if there are $c_2<a_s<b_s<c_1$ and $n\in\N$ such that $T^n|_{[c_2,a_s]}$ and $T^n|_{[a_s, b_s]}$ are one-to-one and $T^n(a_s)=c_2, T^n(b_s)=c$, then $T^{n-2}$ maps $[c_2, b_s]$ one-to-one onto $[z_1, c_n]$ if $c<c_n$, or onto $[c_n, \hat z_1]$, if $c_n<c$. In any case, since there is an interval $[c,b_{-2}]$ mapped one-to-one onto $[c_2, b_s]$, we conclude that $n$ is a cutting time. Since $s\in\cA$, it follows that $\{c_{S_k+2}: k\in\N\}$ is dense in $[c_2, c]$.
	
	Define a map
	\begin{equation}\label{eq:F}
	F:[c_2, c_1] \setminus \{ c \} \to [c_2, c_1] \setminus \{ c \}, \quad 
	y \mapsto T^{S_k}(y) \text{ if } y \in \Upsilon_k,
	\end{equation}
	for $k\in \N_0$, see Figure~\ref{fig:F}.
	
	\begin{figure}[!ht]
		\centering
		\begin{tikzpicture}[scale=4]
		\draw (0,0)--(0,1)--(1,1)--(1,0)--(0,0);
		\draw[dashed] (0,0.45)--(1,0.45);
		\draw[dashed] (0.45,0)--(0.45,1);
		\node at (0,-0.1) {\small $c_2$};
		\node at (0.45,-0.1) {\small $c$};\draw (0.45,-0.02)--(0.45,0.02);
		\node at (1,-0.1) {\small $c_1$};	
		\node at (0.1,-0.1) {\small $z_0$};\draw (0.1,-0.02)--(0.1,0.02);
		\node at (0.81,-0.1) {\small $\hat z_0$};\draw (0.8,-0.02)--(0.8,0.02);	
		\node at (0.25,-0.1) {\small $z_1$};\draw (0.25,-0.02)--(0.25,0.02);
		\node at (0.66,-0.1) {\small $\hat z_1$}; \draw (0.65,-0.02)--(0.65,0.02);	
		\draw[thick] (0.1,0.45)--(0,0.2);
		\draw[thick] (0.8,0.45)--(1,0);
		\draw[thick] (0.25,0.45)--(0.1,1);
		\draw[thick] (0.65,0.45)--(0.8,1);
		\draw[thick] (0.32,0.45)--(0.25,1);
		\draw[thick] (0.58,0.45)--(0.65,1); 
		\draw[thick] (0.35,0.45)--(0.32,0);
		\draw[thick] (0.55,0.45)--(0.58,0);  
		\draw[thick] (0.35,0.2)--(0.37,0.45);
		\draw[thick] (0.55,0.2)--(0.53,0.45);  
		\draw[thick] (0.37,0.6)--(0.38,0.45);
		\draw[thick] (0.53,0.6)--(0.52,0.45); 
		\node at (-0.07, 0.2) {\small $c_3$};
		\node at (-0.07, 0.6) {\small $c_5$};
		\draw (-0.01,0.6)--(0.01,0.6);
		\draw (0.32,-0.01)--(0.32,0.01);	
		\draw (0.35,-0.01)--(0.35,0.01);	
		\draw (0.37,-0.01)--(0.37,0.01);	
		\draw (0.38,-0.01)--(0.38,0.01);	
		\draw (0.58,-0.01)--(0.58,0.01);	
		\draw (0.55,-0.01)--(0.55,0.01);	
		\draw (0.53,-0.01)--(0.53,0.01);	
		\draw (0.52,-0.01)--(0.52,0.01);	
		\end{tikzpicture}
		\caption{The map $F$ for $\nu=1.0.0.11.101.10010\ldots$}
		\label{fig:F}
	\end{figure}
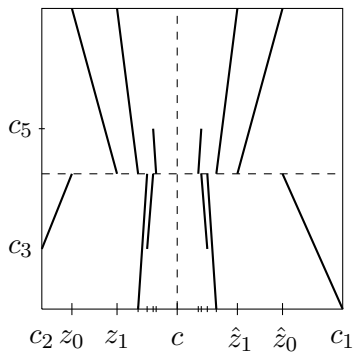
	
	By Equation~\ref{eq:zzz}, $c_{S_k}\in \Upsilon_{Q(k+1)}$, so it follows that $F(c_{S_k})=T^{S_{Q(k+1)}}(c_{S_k})=c_{S_{k+1}}$ for every $k\in\N_0$.

	Let $P := \overline{ \{ c_{S_k} \}}_{k \in\N_0} $ and recall that $P\neq[c_2, c_1]$ by assumption. Thus if $y\in P\setminus\{z_k: k\in\N_0\}$, then $F(y)\in P$.
	Assume by contradiction that $P$ contains an interval $J\subsetneq[c_2, c_1]$. Without loss of generality we can 
	take $J=(y, z_{k})$ for some $y\in(z_{k-1}, z_{k})$ (otherwise iterate and use 
	that $F(c_{S_k})=c_{S_{k+1}}$). 
	Then $F(J)=(T^{S_k}(y), c)$ and $F(J)\subset P$.
	It follows from \cite[Proposition 6.2.12]{BB} that $\omega(c)$ is nowhere dense if
	$\liminf_{k\geq 0} Q(k) \geq 2$, so we can assume that for every $\eps>0$ the interval $(c-\eps, c)$ contains $z_k$ 
	such that $F^2(J) \supset F([z_k, z_{k+1})) = [c_2, c)$ or $(c, c_1]$. We can further conclude that $[c_2, c]$ or $[c, c_1]$ is contained in $P$. But then  $P\supset F^2(P)=[c_2, c_1]$, which is a contradiction.
	
	We conclude that $P$ is nowhere dense and thus $T^2(P)\supset [c_2,c]$ is also nowhere dense, which is a contradiction. 
\end{proof}


\begin{thebibliography}{A}
	\bibitem{Al} L.\ Alvin, {\em Hofbauer towers and inverse limit spaces}, Proc.\ Amer.\ Math.\ Soc.\ {\bf 141} (2013), 4039--4048.
	
	\bibitem{AlBr} L.\ Alvin, K.\ Brucks, {\em Adding machines, kneading maps, and endpoints}, Topology\ Appl.\ {\bf 158} (2011), 542--550.
	
	
	
	\bibitem{BBD} M.\ Barge, K.\ Brucks, B.\ Diamond, {\em Self-similarity of inverse limits of tent maps}, Proc. Amer. Math. Soc.,\ {\bf 124} (1996), 3563--3570.
	
	
	
	
	\bibitem{BaHo} M.\ Barge, S.\ Holte, {\em Nearly one-dimensional H\'{e}non attractors and inverse limits}, Nonlinearity\ {\bf 8} (1995), 29--42.
	
	
	
	\bibitem{BaMa} M.\ Barge, J. Martin, {\em Endpoints of inverse limit spaces and dynamics}, Continua (Cincinnati, OH, 1994), volume {\bf 170} of Lecture Notes in Pure and Appl. Math., 165--182. Dekker, New York, 1995.
	
	
	\bibitem{BaMa2} M.\ Barge, J.\ Martin, {\em Construction of global attractors}, Proc.\ Am.\ Math.\ Soc. {\bf 110} (1990), 523--525.
	
	\bibitem{BKM} L.\ Block, J.\ Keesling, M.\ Misiurewicz, {\em Strange adding machines}, Ergod.\ Th.\ \& Dynam.\ Sys.\ {\bf 26} (2006), 673--682.
	
	\bibitem{BdCH1} P.\ Boyland, A.\ de Carvalho, T.\ Hall, {\em Inverse limits as attractors in parametrized families}, Bull. Lond. Math. Soc. {\bf 45}, no. 5 (2013), 1075--1085.
	
	\bibitem{BdCHInvent}  P.\ Boyland, A.\ de Carvalho, T.\ Hall, {\em New rotation sets in a family of torus homeomorphisms}
	Inventiones Mathematicae, {\bf 204} (3), 895--937, 2016.
	
	\bibitem{BdCH2} P.\ Boyland, A.\ de Carvalho, T.\ Hall, {\em Itineraries for Inverse Limits of Tent Maps: a Backward View}, Topology Appl.\, {\bf 232} (2017), 1--12.
	
	\bibitem{3G} P.\ Boyland, A.\ de Carvalho, T.\ Hall, {\em Natural extensions of unimodal maps: prime ends of planar embeddings and semi-conjugacy to sphere homeomorphisms}, arXiv:1704.06624 [math.DS] (2017).
	
	\bibitem{BdCH} P.\ Boyland, A. de Carvalho, T.\ Hall, {\em Typical path components in tent map inverse limits,} Preprint 2017,  arXiv:1712.00739
	
	\bibitem{BrBr} K.\ Brucks, H.\ Bruin, {\em Subcontinua of inverse limit spaces of unimodal maps}, Fund. Math. {\bf 160} (1999), 219--246.
	
	\bibitem{BB} K.\ Brucks, H.\ Bruin, Topics from one-dimensional dynamics, London Mathematical Society Student Texts {\bf 62}, Cambridge University Press, 2004.
	
	
	
	\bibitem{GreenBook} H.\ Bruin, {\em Invariant measures for interval maps}, PhD thesis, Delft (1994).
	
	\bibitem{Br0} H.\ Bruin, {\em Combinatorics of the kneading map,} Int.\ Jour.\ of Bifur.\ and Chaos, {\bf 5} (1995), 1339--1349. 
	
	
	\bibitem{Br5} H.\ Bruin, {\em Topological conditions for the existence of Cantor attractors}, Trans.\ Amer.\ Math.\ Soc.\ {\bf 350} (1998) 2229-2263. 
	
	
	\bibitem{Br1} H.\ Bruin, {\em Planar embeddings of inverse limit spaces of unimodal maps}, Topology Appl. \ {\bf 96} (1999), 191--208.
	
	
	\bibitem{Br3} H.\ Bruin, {\em Subcontinua of Fibonacci-like unimodal inverse limit spaces}, Topology Proceedings {\bf 31} (2007) no. 1, 37-50.
	
	
	
	
	
	
	\bibitem{Hen} M. H\'enon, {\em A two-dimensional mapping with a strange
	attractor}, Comm. Math. Phys. {\bf 50} (1976) 69--77.

	
	\bibitem{GoKnRa} C.\ Good, R.\ Knight, B.\ Raines, {\em Countable inverse limits of postcritical $\omega$-limit sets of unimodal maps}, Discrete Contin.\ Dynam.\ Systems {\bf 27} (2010), 1059--1078.
	
	\bibitem{Gu} J.\ Guckenheimer, {\em The strange, strange attractor} in J.\ E.\ Marsden, M.\ McCracken, {\em The Hopf bifurcation and its applications} (Springer Lecture Notes in Applied Mathematics), New York (1976), 368--381.
	
	\bibitem{Hof} F.\ Hofbauer,
	{\em The topological entropy of a transformation $x \mapsto ax(1-x)$,}
	Monatsh.\ Math.\ {\bf 90} (1980), 117--141.
	
	
	\bibitem{In} W.\ T.\ Ingram, {\em Inverse limits and dynamical systems}, Open Problems in Topology II, (2007), 289--301.
	
	\bibitem{Lorenz} E.\ Lorenz, {\em Deterministic non-periodic flow}, J.\ Atmosph.\ Sci.\ {\bf 20} (1963), 130--141.
	
	
	\bibitem{Loz} R.\ Lozi, {\em Un attracteur \'etrange du type attracteur de H\'enon}, Journal de Physique.
	Colloque C5, Suppl\'ement au no. 8, {\bf 39}, 9--10, 1978.
	
	\bibitem{Ly} M.\ Lyubich, {\em Combinatorics, geometry and attractors of quadratic-like maps}, Ann. of Math.\ {\bf 6} (1993) 425--457.
	
	\bibitem{dMvS} W.\ de Melo, S.\ van Strien, One-Dimensional Dynamics, Springer, New York, 1993.
	
	\bibitem{Na-Elsa} S.\ B.\ Nadler, Jr., {\em Continua whose cone and hyperspace are homeomorphic}, Trans.\ Amer.\ Math.\ Soc.\ {\bf 230} (1977), 321--345.
	
	\bibitem{Na} S.\ B.\ Nadler,  Continuum Theory: An Introduction, Marcel Dekker, Inc., New York (1992).
	
	
	\bibitem{Raines}
	B.\ Raines, \emph{Inhomogeneities in non-hyperbolic one-dimensional invariant sets}, Fund. Math.\ {\bf 182} (2004), 241--268.
	
	\bibitem{Thun} H.\ Thunberg,
	{\em A recycled characterization of kneading sequences},
	International Journal of Bifurcation and Chaos {\bf 9} (1999), No. 9, 1883--1887. 
	
	\bibitem{Wi} R.F.\ Williams, {\em One-dimensional nonwandering sets,} Topology {\bf 6} (1967), 473--487.
	
\end{thebibliography}
\end{document}